%% file: aos-mixing.tex
\documentclass[11pt,english]{article}
\usepackage{amsmath,amssymb,natbib,graphicx,url,algorithm2e}
\usepackage[colorlinks=true]{hyperref}
\usepackage[T1]{fontenc}
\usepackage{babel}

\RequirePackage{mystyle}

\begin{document}
\title{Lasso Guarantees for $ \beta $-Mixing Heavy Tailed Time Series}

\author{
  Kam Chung Wong\\
%  \thanks{Use footnote for providing further
%    information about author (webpage, alternative
%    address)---\emph{not} for acknowledging funding agencies.} 
  Department of Statistics\\
  University of Michigan\\
  \texttt{kamwong@umich.edu}\\
  \and
  Zifan Li \\
  Departments of Statistics\\
Yale University\\
  \texttt{zifan.li@yale.edu}\\
\and
Ambuj Tewari  \\
Departments of Statistics and EECS        \\
University of Michigan\\
    \texttt{tewaria@umich.edu}\\
}

\maketitle

\begin{abstract}
%\todo{update abstract}
Many theoretical results for the lasso require the samples to be iid. Recent work has provided guarantees for the lasso assuming that the time series is generated by a sparse Vector Auto-Regressive (VAR) model with Gaussian innovations. Proofs of these results rely critically on the fact that the true data generating mechanism (DGM) is a finite-order Gaussian VAR. This assumption is quite brittle: linear transformations, including selecting a subset of variables, can lead to the violation of this assumption. In order to break free from such assumptions, we derive non-asymptotic inequalities for estimation error and prediction error of the lasso estimate of the best linear predictor without assuming any special parametric form of the DGM. Instead, we rely only on (strict) stationarity and geometrically decaying $\beta$-mixing coefficients to establish error bounds for the lasso for subweibull random vectors. The class of subweibull random variables that we introduce includes subgaussian and subexponential random variables but also includes random variables with tails heavier than an exponential. We also show that, for Gaussian processes, the $\beta$-mixing condition can be relaxed to summability of the $\alpha$-mixing coefficients. Our work provides an alternative proof of the consistency of the lasso for sparse Gaussian VAR models. But the applicability of our results extends to non-Gaussian and non-linear times series models as the examples we provide demonstrate.
%Moreover, the inequality can potentially be applied to study other structures such as group sparsity or low rankness plus sparsity, and non-convex penalties such as SCAD and MCP.
\end{abstract}

\input{intro}

\input{prelim}

\input{gaussian-alpha}

\input{subweibull-beta}

\input{simulations}

\subsection*{Acknowledgments}
We thank Sumanta Basu and George Michailidis for helpful discussions, and Roman Vershynin for pointers to the literature. We acknowledge the support of NSF via a regular (DMS-1612549) and a CAREER grant (IIS-1452099).

\bibliographystyle{plain}
\bibliography{myBib,time_series_grant}

\input{appendix}

\end{document}

%% file: intro.tex
%!TEX root = aos-mixing.tex

%\edit{Major edits are in this color.}\\
%\comment{Kam's comments are in this color.}\\
%\ambuj{Ambuj's comments are in this color. Before submission, make sure all comments have been removed. Only edits should stay.}
%
%\comment{@Ambuj:
%(1) Is there any value to fully exploit Liebscher's paper to give broader class of vector ARCH? (2) Should we add some reference on functional dependence? (3) Should we include some of the responses to the reviews in a discussion Section? (4) should we add a conclusion/discussion?
%}
%\ambuj{(1) In the interest of time, no. If the reviewers ask us to do a 2nd revision, we could do this then. (2) Yes, Sumanta had some references in his paper. We should cite the same. Perhaps we already are. But it'll be good to check his paper. (3) No, include responses in the response file only. The manuscript should just have the revised paper with significant changes highlighted in color for reviewers to easily see them. (4) No, unless we have something very important to say that we have not already said.}

\section{Introduction}\label{sec:intro}
%\todo{give reference to limitioans on VAR(1)}
%
%\todo{compre with A Direct Estimation of High Dimensional Stationary Vector Autoregressions}
High dimensional statistics is a vibrant area of research in modern statistics and machine learning~\citep{buhlmann2011statistics,hastie2015statistical}. The interplay between computational and statistical aspects of estimation in high dimensions has led to a variety of efficient algorithms with statistical guarantees including methods based on convex relaxation (see, e.g.,~\cite{chandrasekaran2012convex,negahban2012unified}) and methods using iterative optimization
techniques (see, e.g.,~\cite{beck2009fast,agarwal2012fast,donoho2009message}). However, the bulk of existing theoretical work focuses on iid samples. The extension of theory and algorithms in high dimensional statistics to time series data, where dependence is the norm rather than the exception, is just beginning to occur. We briefly summarize some recent work in Section~\ref{sec:worksummary} below.

Our focus in this paper is to give guarantees for $\ell_1$-regularized least squares estimation, or lasso~\citep{hastie2015statistical}, that hold even when there is temporal dependence in data. The recent work of \cite{basu2015regularized} took a major step forward in providing guarantees for lasso in the time series setting. They considered Gaussian Vector Auto-Regressive (VAR) models with finite lag (see Example~\ref{ex:GaussVAR}) and defined a measure of stability using the spectral density, which is the Fourier transform of the autocovariance function of the time series.
Then they showed that one can derive error bounds for lasso in terms of their measure of stability. Their bounds are an improvement over previous work~\citep{negahban2011estimation,loh2012high,han2013transition} that assumed operator norm bounds
on the transition matrix. These operator norm conditions are restrictive even for VAR models with a lag of $1$ and never hold (Please see pp. 11--13 in the Supplement of \cite{basu2015regularized} for details) if the lag is strictly larger than 1! Therefore, the results of \cite{basu2015regularized} hold in greater generality than previous work. But they do have limitations.

A key limitation is that \cite{basu2015regularized} assume that the VAR model is the true data generating mechanism (DGM). Their proof techniques rely heavily on having the
VAR representation of the stationary process available. The VAR model assumption, while popular in many areas, can be restrictive since the VAR family is not closed under linear transformations: if $Z_t$ is a VAR process and $C$ is a linear transformation then
$C Z_t$ may not be expressible as a finite lag VAR~\citep{lutkepohl2005new}. We later provides examples (Examples~\ref{ex:misVAR} and~\ref{ex:sGmisVAR}) of VAR processes where leaving out a single variable breaks down the VAR assumption. What if we do not assume that $Z_t$ is a finite lag VAR process but simply that it is stationary? Under stationarity (and finite 2nd moment conditions),
the best linear predictor of $Z_t$ in terms of $Z_{t-d},\ldots, Z_{t-1}$ is well defined even if $Z_t$ is not a lag $d$ VAR. If we assume that this best linear predictor involves sparse coefficient matrices, can we still guarantee consistent parameter estimation? Our paper provides an affirmative answer to this important question.

We provide finite sample parameter estimation and prediction error bounds for lasso in two cases: (a) for stationary Gaussian processes with suitably decaying $\alpha$-mixing coefficients (Section~\ref{sect:gaus}), and (b) for stationary processes with subweibull marginals and  geometrically decaying $\beta$-mixing coefficients (Section~\ref{section:subweibull}). It is well known that guarantees for lasso follow if one can establish restricted eigenvalue (RE) conditions and provide deviation bounds (DB) for the correlation of noise with the regressors (see the Master Theorem in Section~\ref{sect:masterThm} below for a precise statement). Therefore, the bulk of the technical work in this paper boils down to establishing, with high probability, that DB and RE conditions hold under the Gaussian $\alpha$-mixing ( Propositions~\ref{result:rhoDB} and~\ref{result:rhoRE}) and the subweibull $\beta$-mixing assumptions respectively (Propositions~\ref{Lemma:SubweibullDevBound} and~\ref{Lemma:SubweibullRE}). Note that RE conditions were previously shown to hold under the \emph{iid assumption} by \cite{raskutti2010restricted} for Gaussian random vectors and by \cite{rudelson2013reconstruction} for subgaussian random vectors. \edit{We also include some simulations (Section~\ref{sect:sim}) to study the effect of VAR dimension, tail behavior, and temporal dependence on the estimation error decay rate as a function of the sample size.}
%[deleted]  Our results in the sub-Gaussian case rely on novel concentration inequality (Lemma~\ref{result: subgau:tailbdd}) for $\beta$-mixing sub-Gaussian random variables that may be of independent interest. All proofs are deferred to the appendix.

\subsection{Summary of Recent Work on High Dimensional Time Series}\label{sec:worksummary}
While we discussed the work of~\cite{basu2015regularized} -- since ours is closely related to theirs --
we wish to emphasize that several other researchers have recently published work on statistical analysis of high dimensional time series.
\cite{song2011large}, \cite{wu2015high} and \cite{alquier2011sparsity} give theoretical guarantees assuming that RE conditions hold. As~\cite{basu2015regularized} pointed out, it takes a fair bit of work to actually establish RE conditions in the presence of dependence. \cite{chudik2011infinite,chudik2013econometric,chudik2014theory} use high dimensional time series for global
macroeconomic modeling. Alternatives to lasso that have been explored include quantile based methods for heavy-tailed data~\citep{qiu2015robust}, quasi-likelihood approaches~\citep{uematsu2015penalized}, two-stage estimation techniques~\citep{davis2016sparse} and the Dantzig selector~\citep{han2013transition, han2015direct}.  Both \cite{han2013transition} and \cite{han2015direct} studied the stable  Gaussian VAR models while our paper covers wider classes of processes as our examples demonstrate.
{ \cite{fan2016penalized} considered the case of multiple sequences of univariate $ \alpha $-mixing heavy-tailed dependent data. Under a stringent condition on the auto-covariance structure (please refer to Appendix~\ref{proof:subweill} for details), the paper established finite sample $ \ell_2 $ consistency in the real support for penalized least squares estimators. In addition, under mutual incoherence type assumption,
%     (TODO: did they establish mutual incoherence?)
it provided sign and $ \ell_\infty $ consistency. An AR(1) example was given as an illustration.} 
% TODO: Kam, say more here.
Both~\cite{uematsu2015penalized} as well as~\cite{kock2015oracle} establish oracle inequalities for the lasso applied to time series prediction. \cite{uematsu2015penalized} provided results not just for lasso but also for estimators using penalties such as the SCAD penalty. Also, instead of assuming Gaussian errors, it is only assumed that fourth moments of the errors exist. \cite{kock2015oracle} provided non-asymptotic lasso error and prediction error bounds for stable Gaussian VARs. Both \cite{sivakumar2015beyond} and \cite{medeiros2016} considered subexponential designs. \cite{sivakumar2015beyond} studied lasso on iid subexponential designs and provide finite sample bounds. \cite{medeiros2016} studied adaptive lasso for linear time series models and provide sign consistency results.
\cite{wang2007regression} provided theoretical guarantees for lasso in linear regression models with autoregressive errors.
Other structured penalties beyond the $\ell_1$ penalty have also been considered \citep{nicholson2014hierarchical,nicholson2015varx,guo2015high,ngueyep2014large}.
\cite{zhang2015gaussian},~\cite{mcmurry2015high},~\cite{wang2013sparse}
and~\cite{chen2013covariance} consider estimation of the covariance (or precision) matrix of high dimensional time series.
\cite{mcmurry2015high} and~\cite{nardi2011autoregressive} both highlight that autoregressive (AR) estimation, even in univariate time series, leads to high dimensional parameter estimation problems if the lag is allowed to be unbounded.

%\red{please check the following added subsection is good}
\subsection{Organization of the Paper}
Section~\ref{section:prelim} introduces our notation, presents the assumptions used to derive our key results, and states some useful facts needed later. Then we present two sets of high probability guarantees for the lower restricted eigenvalue and deviation bound conditions in Sections~\ref{sect:gaus} and \ref{section:subweibull} respectively. 
%Section~\ref{sect:subgauss} covers $ \beta$-mixing  time series with sub-Gaussian observations and all parameter dependences explicit; 
Section~\ref{sect:gaus} covers $ \alpha $-mixing Gaussian time series. Note that $ \alpha $-mixing is a weaker notion than $ \beta$-mixing and all the parameter dependences are explicit. It is followed by Section~\ref{section:subweibull} which covers $ \beta$-mixing time series with subweibull observations and we make the dependence on the subweibull norm explicit. Section~\ref{sect:sim} presents two simulation results: one where we vary the heaviness of the tail of the random vectors in the time series and another one where we vary the degree of temporal dependence in the time series.

%\red{\todo{fix the descriptions here}}
We present five examples, two involving $\alpha$-mixing Gaussian processes and three $ \beta$-mixing subweibull vectors. They are presented along with the corresponding theoretical results to illustrate applicability of the theory. 
Examples~\ref{ex:GaussVAR} and~\ref{ex:misVAR} concern applications of the results in Section~\ref{sect:gaus}. We consider VAR models with Gaussian innovations when the model is correctly or incorrectly specified.
%Example~\ref{examp:subWeilVAR} concerns VAR models with general subweibull innovations to illustrate applications in heavy tailed scenarios. 
In Examples~\ref{examp:subWeilVAR},~\ref{ex:sGmisVAR}, and~\ref{ex:ARCH}, we focus on
the case of subweibull random vectors. We consider VAR models with subweibull innovations when the model is correctly or incorrectly specified (Examples~\ref{examp:subWeilVAR} and ~\ref{ex:sGmisVAR}). In addition, we go beyond linear models and introduce non-linearity in the DGM in Example~\ref{ex:ARCH}.

These examples serve to illustrate that our theoretical results for lasso on high dimensional dependent data estimation extend beyond the classical linear Gaussian setting and provides guarantees potentially in the presence of one or more of the following scenarios: model mis-specification, heavy tailed non-Gaussian innovations and nonlinearity in the DGM.

%% file: prelim.tex
\section{Preliminaries}\label{section:prelim}
%\comment{
%Reviewer 1 thinks this section is verbose. But i think this is helpful fo rresders unfamiliar with the materials. So, I did not cut it. Please let me know what you think.
%}\\
%\ambuj{Yeah, we can keep it.}

Consider a stochastic process of pairs $ (X_t, Y_t)_{t=1}^\infty $ where
$ X_t\in \R^p ,\, Y_t\in \R^q,\, \forall t $. One might be interested in predicting $ Y_t$ given $ X_t $. In particular, given a dependent sequence  $(Z_t)_{t=1}^T$, one might want to forecast the present $Z_t$ using the past $ (Z_{t-d},\ldots,Z_{t-1})$. A linear predictor is a natural choice. To put it in the regression setting, we identify  $Y_t = Z_t$ and $X_t = (Z_{t-d},\ldots,Z_{t-1})$. The pairs $(X_t, Y_t)$ defined as such are no longer iid. Assuming strict stationarity, the parameter matrix of interest $\bstar \in \R^{p \times q} $ is
\begin{equation} \label{eqn:bstar}
\bstar = \argmin_{\bb \in \R^{p \times q}} \E [ \vertii{ Y_t - \bb' X_t}_2^2  ] .
\end{equation}
Note that $\bstar$ is independent of $t$ owing to stationarity. Because of high dimensionality ($ pq \gg T $), consistent estimation is impossible without regularization. We consider the lasso procedure. The  $\ell_1$-penalized least squares estimator $\bhat \in \R^{p \times q} $ is defined as
\begin{equation} \label{eqn:bhat}
\bhat = \argmin_{\bb \in \R^{p \times q}} \frac{1}{T}\Vert \vect( \mt{Y}-\mt{X}\bb ) \Vert_2^2+ \lambda_T \vertii{\vect(\bb)}_1  .
\end{equation}
where
\begin{align}\label{eq:XYdef}
\mt{Y} &=(Y_1,Y_2,\,\ldots\,, Y_T)' \in \R ^{T \times q} 
&
\mt{X} &=(X_1,X_2,\,\ldots\,, X_T)' \in \R ^{T \times p} .
\end{align}
The following matrix of true residuals is not available to an estimator but will appear in our analysis:
\begin{align}\label{eq:Wdef}
\mt{W} &:= \mt{Y} - \mt{X} \bstar.
\end{align}

%\green{Similar theoretical results hold for the corresponding constrained version of the estimators.} \green{include them???} \\

\subsection{Notation}
For scalars $ a $ and $ b $, define shorthands  $a \wedge b:=\min\{a,b\}$ and $a \vee b :=\max\{a,b\}$. For a symmetric matrix $ \mt{M} $, let $\lmax{\mt{M}}$ and $\lmin{\mt{M}}$ denote its maximum and minimum eigenvalues respectively.
\edit{ 
For any square matrix $ \mt{M}  $ with rank $ d $, let $ \lambda_i{\mt(M)},\, i = 1, \dots, d $ denote its eigenvalues.  Then, $\sr{\mt{M}}$ denotes its spectral radius $\max_i{\{|\lambda_i(\mt{M})|\}}$.
}
For any matrix $\mt{M}$, let  $\vertiii{\mt{M}}$, $\vertiii{\mt{M}}_\infty$, and $\vertiii{\mt{M}}_F$ denote its operator norm $\sqrt{\lambda_{\max}(\mt{M}'\mt{M})}$, entry-wise $\ell_\infty$ norm $\max_{i,j} |\mt{M}_{i,j}|$, and Frobenius norm $\sqrt{\mathrm{tr}(\mt{M}'\mt{M})}$ respectively. For any vector $v\in \R^p$, $\vertii{v}_q$ denotes its $\ell_q$ norm $ (\sum_{i=1}^p |v_i|^q)^{1/q}$. 
%\red{TODO: sometimes vectors are bold, sometimes note. We need to be consistent.} 
Unless otherwise specified, we shall use $\vertii{\cdot}$ to denote the $\ell_2$ norm. For any vector $\vc{v} \in \R^p $, we use $\vertii{\vc{v}}_0$ and  $\vertii{\vc{v}}_{\infty}$ to denote $\sum_{i=1}^{p} \mathbbm{1} \{\vc{v}_i\neq 0\}$ and $\max_i\{|\vc{v}_i|\}$ respectively. Similarly, for any matrix $ \mt{M}$, $ \vertiii{\mt{M}}_{0}= \vertii{\vect(\mt{M})}_0$ where $\vect(\mt{M})$ is the vector obtained from $\mt{M}$ by concatenating the rows of $M$. We say that matrix $ \mt{M} $ (resp. vector $ \vc{v} $) is \textit{$ s $-sparse} if $ \vertiii{\mt{M}}_0=s$ (resp. $ \vertii{ \vc{v}}_0 =s $). We use $ \vc{v}' $ and $ \mt{M}' $ to denote the transposes of  $ \vc{v} $ and $ \mt{M} $ respectively. 
When we index a matrix, we adopt the following conventions. For any matrix $ \mt{M}\in \R^{p\times q} $, for $ 1\le i \le p$, $1\le j\le q $, we define  $  \mt{M}[i,j]\equiv\mt{M}_{ij}:=\vc{e}_i'\mt{M}\vc{e}_j $, $\mt{M}[i,:]\equiv\mt{M}_{i:}:=\vc{e}_i'\mt{M} $ and $ \mt{M}[:,j]\equiv\mt{M}_{:j}:=\mt{M}\vc{e}_j $ where $ \vc{e}_i  $  is the vector with all $ 0 $s except for a $ 1 $ in the $ i $th coordinate.
The set of integers is denoted by $\mathbb{Z}$. \edit{
For simplicity, $ \Sigma $ and $ \Gamma $ are not in \textbf{bold} font in this paper. 
}\\
%\ambuj{Why do we need to mention this here?}
%\red{TODO: define what $a \wedge b$ and $a \vee b$ mean.}

%For a stationary sequence mean zero process $(Z_t )_{t=1}^{T}$, $Z_t \in \R^p,\, \forall t$, denote the auto-covariance matrix at lag $l$ as $\Sigma_Z(l):= \mathbb{E}[Z_t {Z}_{t+l}'], \forall t,\, l \in \mathbb{Z}$. 
%Further, let $ \Lambda_{Z}:=[\Sigma_{Z}(|i-j|)]_{1\le i,j \le T} $ be a huge $Tp \times Tp$ matrix whose $ (i,j) $th block is the $ p \times p $ auto-covariance matrix of $ ({Z}_t)_{t=1}^T $ at lag $ |i-j| $.
% For any two stationary mean zero processes, $ {Z_t} $ and $ {W}_t $, let $ \Sigma_{{Z},{W}}(l) := \E \bbra{{Z_t}{W_{t+l}}'} $ denote the cross-covariance matrix at lag $l$. Finally we will often abbreviate the covariance matrix $\Sigma_Z(0) = \E[Z_t Z_t']$ to $\Sigma_Z$ and $\Sigma_{Z,W}(0) = \E[Z_tW_t']$ to $\Sigma_{Z,W}$.
%\red{\todo{check subsetting matrix notations}}
For a lag $ l \in \mathbb{Z}$, we define the auto-covariance matrix w.r.t. $ (X_t, Y_t)_t $ as $\Sigma(l) = \Sigma_{({X;Y})}(l):=\E [(X_t;Y_t)(X_{t+l};Y_{t+l})'] $. Note that $\Sigma(-l) = \Sigma(l)'$. Similarly, the auto-covariance matrix of lag $ l $ w.r.t. $(X_t)_t$ is $ \Sigma_{{X}}(l):=\E[ X_tX_{t+l}']$, and w.r.t. $(Y_t)_t$ is $ \Sigma_{Y}(l):=\E [Y_t Y_{t+l}' ]$. At lag $ 0 $, we often simplify the notation as $ \Sigma_X \equiv \Sigma_X(0) $ and $ \Sigma_Y \equiv \Sigma_Y(0) $. 

The cross-covariance matrix at lag $l$ is $ \Sigma_{X,Y}(l):=\E[ X_t Y_{t+l}' ]$. Note the difference between $\Sigma_{(X;Y)}(l)$ and $\Sigma_{X,Y}(l)$: the former is a $(p+q) \times (p+q)$ matrix, the latter is a $p \times q$ matrix. 
\edit{
Thus, $ \Sigma_{(X;Y)}(l)$ is a matrix consisting of four sub-matrices with the following block structure:
\begin{equation*}
\Sigma_{(X;Y)}(l) = \begin{bmatrix}
\Sigma_{X}(l) & \Sigma_{X,Y}(l) \\
\Sigma_{Y,X}(l) & \Sigma_{Y}(l)  
\end{bmatrix}.
\end{equation*}
}

\edit{
Let $\mathbf{1}$ and $ \mathbf{0}$ denote vectors consisting of ones and zeros respectively with dimensionality indicated in a subscript (if it is not clear from the context).}
We adopt the convention that, at lag $ 0 $, we omit the lag argument $ l $. For example, $ \Sigma_{X,Y}$ denotes $\Sigma_{X,Y}(0) = \E[ X_t Y_t' ]$. Finally, let $ \hat{\Gamma}:= \frac{\mt{X}'\mt{X}}{T} $ be the empirical covariance matrix.

%\subsection{Sub-Gaussian Constants for Random Vectors}
%
%We would like to consider broader class of random vectors than Gaussian and yet maintain the nice thin tail property of the distribution. One such nice family is that of sub-Gaussian distributions. They are characterized by having tail probabilities of the same or lower order as Gaussian. They have various equivalent definitions, we adopt the following from \cite{rudelson2013hanson}. 
%
%\begin{defn}[Sub-Gaussian Norm and Random Variables/Vectors]
%A random variable $ U $ is called sub-Gaussian with sub-Gaussian constant $K$ if its sub-Gaussian norm
%$$
%\snorm{U}  := \sup_{p \ge 1} p^{-\half} (\E\verti{U}^p)^{1/p}
%$$
%satisfies $\snorm{U} \le K$.
%
%A random vector $ V \in \R^n  $ is called sub-Gaussian if all of its one-dimensional projections are sub-Gaussian and we define 
%$$ \snorm{{V}}:= \sup_{\vc{v}\in \R^n:\vertii{\vc{v}}\le 1}\snorm{\vc{v}'{V}} $$.
%\end{defn}
%
%\begin{defn}[Sub-exponential Norm and Random Variables/Vectors]
%A random variable $ U $ is called sub-exponential with sub-exponential constant $K$ if its sub-exponential norm
%$$
%\enorm{U}  := \sup_{p \ge 1} p^{-1} (\E\verti{U}^p)^{1/p}
%$$
%satisfies 
%$\enorm{U} $ $\le K$.
%
%A random vector $ V\in \R^n  $ is called sub-exponential if all of its one-dimensional projections are sub-exponential and we define 
%$$ \enorm{{U}}:= \allowbreak \sup_{\vc{v}\in \allowbreak \R^n :\allowbreak\vertii{\vc{v}}\le 1}\enorm{\vc{v}'{V}} $$.
%\end{defn}

\subsection{Sparsity, Stationarity and Zero Mean Assumptions}

The following assumptions are maintained throughout; we will make additional assumptions specific to each of the subweibull and Gaussian scenarios. 
Our goal is to provide finite sample bounds on the error $\bhat - \bstar$. We shall present theoretical guarantees on the $\ell_2$ parameter estimation error  $\| \vect(\bhat -\bstar) \|_2$ and also
the associated (in-sample) prediction error $\vertiii{ (\bhat-\bstar)' \hat{\Gamma} (\bhat-\bstar)  }_F$. 

\begin{assump}\label{as:sparse}
The matrix $\bstar$ is $s$-sparse; i.e., $\vertii{\vect(\bstar)}_0 = s$. \label{as:spars}
\end{assump}
\begin{assump} \label{as:stat}
The process $(X_t, Y_t)$ is strictly stationary; i.e., $ \forall m, \tau,\, n \ge 0$,
\[
((X_{m},Y_{m}),\cdots ,(X_{m+n},Y_{m+n}))~\overset{d}{=} ~((X_{m+\tau},Y_{m+\tau}),\cdots,(X_{m+n+\tau},Y_{m+n+\tau})) .
\] where ``$\overset{d}{=}$'' denotes equality in distribution. \label{as:stat}
\end{assump}
\begin{assump} \label{as:0mean}
The process $(X_t, Y_t)$ is centered; i.e., $\forall t,\ \E(X_t)=\mathbf{0}_{p \times 1}, $
and
$\E(Y_t)=\mathbf{0}_{q \times 1}$ .
\end{assump}

\subsection{A Master Theorem}\label{sect:masterThm}
We shall start with what we call a ``master theorem" that provides non-asymptotic guarantees for lasso estimation and prediction errors under two well-known conditions, viz., the restricted eigenvalue (RE) and the deviation bound (DB)
conditions.
Note that in the classical linear model setting (see, e.g.,~\cite[Ch 2.3]{hayashi2000econometrics}) where sample size is larger than the dimensionality
% ($n>p$)
, the conditions for consistency of the ordinary least squares(OLS) estimator are as follows:
(a) the empirical covariance matrix $\mt{X}'\mt{X}/T \overset{P}{\rightarrow}\mt{Q}$ and $ \mt{Q} $ invertible; i.e., $\lmin{\mt{Q}}>0$, and (b)
the regressors and the noise are asymptotically uncorrelated; i.e., $\mt{X}'\mt{W} /T\rightarrow \mt{0}$.

In high-dimensional regimes,~\cite{bickel2009simultaneous},~\cite{loh2012high} and~\cite{negahban2012restricted} have established similar consistency conditions for lasso. The first one is the \textit{restricted eigenvalue} ({RE}) condition on $\mt{X}'\mt{X}/T$ (which is a special case, when the loss function is the squared loss, of the \textit{restricted strong convexity} ({RSC}) condition). The second is the \textit{deviation bound} ({DB}) condition on $\mt{X}'\mt{W}/T$.
%\green{modify it to accomodate approx sparsity via a lemma that links RE to that of the approx sparse version}
The following lower    {RE} and    {DB} definitions are modified from those
given by \cite{loh2012high}. 

\begin{defn}[Lower Restricted Eigenvalue]\label{defn:RE}
A symmetric matrix ${\Gamma}\in \R^{p\times p} $ satisfies a lower restricted eigenvalue condition with curvature 
$\edit{ \alphac } >0$  and tolerance $\tau(T,p)>0$ if,
$$
\forall \vc{v} \in \R^{p},\ \vc{v}' {\Gamma}\vc{v} \ge \alphac \vertii{\vc{v}}^2_2 - \tau(T,p)\vertii{\vc{v}}^2_1 .
$$
\end{defn}

\begin{defn}[Deviation Bound]\label{defn:DB}
Consider the random matrices $\mt{X} \in \R^{T\times p}$ and $\mt{W}\in \R^{T\times q}$ defined in~\eqref{eq:XYdef} and~\eqref{eq:Wdef} above. They are said to satisfy the deviation bound condition if there exist a deterministic multiplier function $ \mathbb{Q}(\mt{X},\mt{W},\bstar)$ and a rate of decay function $\mathbb{R}(p,q,T)$ such that,
$$
\frac{1}{T}\vertiii{\mt{X}'\mt{W}}_{\infty} \le \mathbb{Q}(\mt{X},\mt{W},\bstar) \mathbb{R}(p,q,T) .
$$
\end{defn}

%\textcolor{green}{add theorem with approx. sparsity}\\
%\green{Actually can accomodate missingness and data corruption}\\
%\green{Include support recovery resutls???}

We now present a master theorem that provides guarantees for the $\ell_2$ parameter estimation error and the (in-sample) prediction error. The proof, given in Appendix~\ref{sec:masterproof}, builds on existing result of the same kind~\citep{bickel2009simultaneous,loh2012high,negahban2012restricted} and we make no claims of originality for either the result or the proof. 

\begin{thm}[Estimation and Prediction Errors] \label{result:master}
Consider the lasso estimator $\bhat$ defined in \eqref{eqn:bhat}. Suppose Assumption~\ref{as:spars} holds. Further, suppose that $\hat{\Gamma}  :=\mt{X}'\mt{X}/T$ satisfies the lower RE$(\alphac, \tau)$ condition with $\alphac \ge 32s\tau$ and  $\mt{X}'\mt{W}$ satisfies the deviation bound. Then, for any  $\lambda_T\ge 4 \mathbb{Q}(\mt{X},\mt{W},\bstar)\mathbb{R}(p,q,T) $, we have the following guarantees:
\begin{eqnarray}
\vertii{\vect(\bhat-\bstar)} \le 4\sqrt{s}\lambda_T/\alphac ,			\label{eq:l2errorBdd}
\\
\vertiii{ (\bhat-\bstar)' \hat{\Gamma} (\bhat-\bstar)  }_F^2
\le
\frac{32\lambda_T^2 s}{\alphac}			.							\label{eq:predErrBdd}
\end{eqnarray}
\end{thm} 

With this master theorem at our disposal, we just need to establish the validity of the restricted eigenvalue ({RE}) and deviation bound ({DB}) conditions for stationary time series  by making appropriate assumptions.
We shall do that \emph{without} assuming any parametric form of the data generating mechanism. Instead, we will impose appropriate tail conditions on the random vectors $X_t,Y_t$ and also assume that they satisfy some type of mixing condition. Specifically, in Section~\ref{sect:gaus}, we consider $\alpha$-mixing Gaussian random vectors.
 Next, in Section~\ref{section:subweibull}, we consider $\beta$-mixing subweibull random vectors (we define subweibull random vectors below in Section~\ref{section:subweibulldef}). Historically, mixing conditions were introduced
to generalize the classic limit theorems in probability beyond the case of iid random variables~\citep{rosenblatt1956central}. Recent work on high dimensional statistics has established the validity of {RE} conditions in the iid Gaussian~\citep{raskutti2010restricted} and iid subgaussian cases~\citep{rudelson2013reconstruction}. One of the main contributions of our work is to extend these results in high dimensional statistics from the iid to the mixing case.

\edit{
\subsection{Proof Strategies for the RE and DB bounds}
\input{proof_strategy}
}
\subsection{A Brief Overview of Mixing Conditions}\label{sec:mixingintro}

Mixing conditions~\citep{bradley2005basic} are well established in the stochastic processes literature as a way to allow for dependence in extending results from the iid case. The general idea is to first define a measure of dependence
between two random variables $X,Y$ (that can be vector-valued or even take values in a Banach space) with associated sigma algebras $\sigma(X), \sigma(Y)$. For example, 
\[
\alpha(X,Y) = \sup \{ |P(A \cap B) - P(A) P(B)| \::\: A \subset \sigma(X), B \subset \sigma(Y) \} .
\]
Then for a stationary stochastic process $(X_t)_{t=-\infty}^{\infty}$, one defines the mixing coefficients, for $l \ge 1$,
\[
\alpha(l) = \alpha(X_{-\infty:t}, X_{t+l:\infty}) .
\]
We say that a process is mixing, in the sense just defined, when $\alpha(l) \to 0$ as $l \to \infty$. The particular notion we get using the $\alpha$ measure of dependence above is called ``$\alpha$-mixing".
It was first used by~\cite{rosenblatt1956central} to extend the central limit theorem to dependent random variables. There are other, stronger notions of mixing, such as $\rho$-mixing and $\beta$-mixing that are defined using the
dependence measures:
\begin{align*}
\rho(X,Y) &= \sup \{ \mathrm{Cov}(f(X), g(Y)) \::\: \E{f} = \E{g} = 0, \E{f^2} = \E{g^2} = 1 \} \\
\beta(X,Y) &= \sup \frac{1}{2} \sum_{i=1}^I \sum_{j=1}^J | P(S_i \cap T_j) - P(S_i)P(T_j) | 
\end{align*}
where the last supremum is over all pairs of partitions $\{A_1,\ldots,A_I\}$ and $\{B_1,\ldots,B_I\}$ of the sample space $\Omega$ such that $A_i \in \sigma(X), B_j \in \sigma(Y)$ for all $i,j$.
The $\rho$-mixing and $\beta$-mixing conditions do not imply each other but each, by itself, implies $\alpha$-mixing~\citep{bradley2005basic}. For stationary gaussian processes, $\rho$-mixing is equivalent to $\alpha$-mixing (see Fact~\ref{fact:alpharhoEquiv} below).

The $\beta$-mixing condition has been of interest in statistical learning theory for obtaining finite sample generalization error bounds for empirical risk minimization~\citep[Sec. 3.4]{vidyasagar2003learning} and boosting~\citep{kulkarni2005convergence}
for dependent samples. There is also work on estimating $\beta$-mixing coefficients from data~\citep{mcdonald2011estimating}. The usefulness of $\beta$-mixing lies in the fact that by using a simple blocking technique, that goes back to the work
of~\cite{yu1994rates}, one can often reduce the situation to the iid setting. At the same time, many interesting processes such as Markov and hidden Markov processes satisfy a $\beta$-mixing
condition~\citep[Sec. 3.5]{vidyasagar2003learning}. To the best of our knowledge, however, there are no results showing that RE and DB conditions holds under mixing conditions. Next we fill this gap in the literature. Before we continue, we note
an elementary but useful fact about mixing conditions, viz., they persist under arbitrary measurable transformations of the original stochastic process.

%\comment{@Ambuj: 
%Not sure what was the purpose of this fact \ref{fact: bddMixing}. It seems redundant and this is where reviewer one complains that the arguments should have been random variables instead of sigma algebra. I would prefer to remove this fact to solve the porblem since I do not think we get to apply it anywhere in our proofs. 
%}\\
%\ambuj{Yes, we should get rid of this fact if we don't use it.}
%
%%TODO: add fact about boundedness of mixing coefficients
%\begin{fact}\label{fact: bddMixing}
%The range of values that the $ \alpha $, $ \beta $ and $ \rho $-mixing coefficients can take on are bounded(see e.g. \cite{bradley2005basic}):
%Consider the probability space $ (\Omega, \mathcal{F}, P) $, for any two sigma fields $ \mathcal{A}, \mathcal{B}\in \mathcal{F} $, we have
%\begin{align*}
%0\le \alpha(\mathcal{A}, \mathcal{B})\le 1/4, && 0\le \beta(\mathcal{A}, \mathcal{B})\le 1, && 0\le \rho(\mathcal{A}, \mathcal{B})\le 1
%\end{align*}
%\end{fact}

\begin{fact}\label{fact:mixingEquiv}
Suppose a stationary process $ \{U_t\}_{t=1}^T $ is $\alpha$, $ \rho$, or $\beta$-mixing. Then the  stationary sequence $ \{f(U_t)\}_{t=1}^T $, for any measurable function $ f(\cdot)$, also is mixing in the same sense with its mixing coefficients bounded by those of the original sequence. 
\end{fact}

%% file: proof_strategy.tex
%!TEX root = aos-mixing.tex

The key ingredients in establishing both the DB and RE conditions are concentration inequalities. The general strategy is to discretize the vector space, apply the concentration inequality, and use the union bound. This occurs in the proofs establishing the DB and RE conditions in both cases: $ \alpha $-mixing Gaussian and $ \beta $-mixing subweibull. 

A brief sketch of the proof of the DB condition via concentration goes like this: 
Consider a fixed vector $v \in \R^p$ and let $\Sigma_X = \E[ X_t X_t^T ]$. 
Use concentration inequality to show that
$$
v' \mt{X}' \mt{X} v / T - v' \Sigma_X v = \sum  (1/T) \sum_{t=1}^T \left( \| X_t'v \|_2^2 - \E[ \| X_t'v \|_2^2 ] \right)
$$
is sufficiently small.
Then apply the union bound over a set of sparse $v$.

The arguments to show the RE condition via concentration proceed as follows. Note that
$$
\vertiii{\mt{X}'\mt{W} }_{\infty}= \max_{1\le i \le p,1\le j \le q } | [\mt{X}'\mt{W}]_{i,j} |= \max_{1\le i \le p,1\le j \le q }\verti{ (\mt{X}_{:i})'\mt{W}_{:j}} .
$$
At the population level, there is no correlation between $\mt{W}$ and $\mt{X}$. Therefore,
$$
\mathbb{E} (\mt{X}_{:i})'(\mt{Y}-\mt{X}\bstar)=\vc{0},\forall i \;
\Rightarrow 
\mathbb{E} ({\mt{X}_{:i}})'\mt{W}_{:j}=0,\forall i,j .
$$
Fix $i,j$ and write,
\begin{align*}
\verti{(\mt{X}_{:i})'\mt{W}_{:j}} &= \verti{(\mt{X}_{:i})'\mt{W}_{:j} - \E[ (\mt{X}_{:i})'\mt{W}_{:j} ]}  \\
&\le \half \verti{  \Vert \mt{X}_{:i}+ \mt{W}_{:j} \Vert^2 - \E[\Vert \mt{X}_{:i}+ \mt{W}_{:j} \Vert^2 ]  } \\
&\quad + \half \verti{  \Vert \mt{X}_{:i} \Vert^2 - \E[\Vert \mt{X}_{:i} \Vert^2 ] } + \half \verti{  \Vert \mt{W}_{:j} \Vert^2 -\E[ \Vert \mt{W}_{:j} \Vert^2]  } .
\end{align*}

The Hanson-Wright inequality (Lemma~\ref{thm:hanson}) takes care of the Gaussian process case. For the independent subgaussian case, the classical Bernstein's concentration inequality will allow us to prove lasso guarantees. \editt{However, applying the Bernstein's inequality requires the random vectors to satisfy \emph{independence} and \emph{subexponential} tail assumptions. Since a random variable is subgaussian if and only if its square is subexponential, the set of conditions required for the original stochastic process translate into \emph{independence} and \emph{subgaussian}.}

\editt{Often times, real time series data exhibits large tail behavior in addition to being dependent. Therefore, the analysis of lasso for real life time series data requires the arguments to deal with the two complications.} As a result, we need ways to quantify dependence and heavy tailed behavior. Then we need concentration inequalities that hold under weaker conditions. Next, we quantify \emph{dependence} using \emph{mixing coefficients}. Also, we quantify \emph{tail behavior} using the notion of \emph{subweibull} random variables. The concentration inequality we use here is Lemma~\ref{lemma:convenient form} which we derive in Appendix~\ref{sec:betamixingsubweibull}) building on the work of \cite{merlevede2011bernstein}.

%% file: gaussian-alpha.tex
%!TEX root = aos-mixing.tex

\section{Gaussian Processes under $\alpha$-Mixing} \label{sect:gaus}

Here we will study Gaussian processes under the $ \alpha $-mixing condition which is weaker than that of the $ \beta $-mixing. We make the following additional assumptions.

\begin{assump}[Gaussianity]\label{as:gauss}
The process $(X_t,Y_t)$ is a Gaussian process.
\end{assump}

Assume  $(X_t, Y_t)_{t=1}^T$ satisfies Assumptions~\ref{as:stat},~\ref{as:0mean}, and~\ref{as:gauss}. Note that $X_t \sim \mathcal{N}(\vc{0},\Sigma_X)$ and $Y_t \sim \mathcal{N}(\vc{0},\Sigma_Y)$.
To control dependence over time, we will assume $\alpha$-mixing, the weakest notion among $\alpha$, $\rho$ and $\beta$-mixing.

\begin{assump}[$\alpha$-Mixing]\label{as:alphaMix}
The process $(X_t,Y_t)$ is an $\alpha$-mixing process. Let $ S_{\alpha}(T):=\sum_{l=0}^{T}\alpha(l) $. If $ \alpha(l) $ is summable, we let $ \tilde{\alpha}:=\lim\limits_{T \rightarrow \infty} S_{\alpha}(T)<\infty $.
\end{assump}

We will use the following useful fact~\citep[p. 111]{ibragimov1978gaussian} in our analysis. 
\begin{fact}\label{fact:alpharhoEquiv}
For any stationary Gaussian process, the $\alpha$ and $\rho$-mixing coefficients are related as follows:
$$
\forall l\ge 1,\ \alpha(l) \le \rho(l) \le 2\pi \alpha(l) .
$$
\end{fact}

%\red{TODO: In gaussian case, we give RE before DB. In subweibull case, we give DB before RE. Need to be consistent about this order both in the paper and in the appendix.}

%\subsection{High Probability Guarantee for Deviation Bound}

\begin{pr}[Deviation Bound, Gaussian Case]\label{result:rhoDB}
Suppose Assumptions \ref{as:stat}--\ref{as:alphaMix} hold.
%Consider the M-estimation problem \eqref{eqn:bstar}. 
%Define the residual matrix as $\mt{W}:= \mt{Y}-\mt{X}\bstar $, where $\vc{Y}:= [Y_1',\dots,Y_T']'$ and $\vc{X}:= [X_1',\dots,X_T']'$ and $\bstar \in \R ^{p\times q}$ a deterministic matrix. \red{TODO: change the language here}\\
Then, there exists a deterministic positive constant $\tilde{c}$, and a free parameter \black{$b>0$}, such that, for $\black{T\ge \sqrt{\frac{b+1}{\tilde{c}}} \log({pq})}$, we have
$$
\mathbb{P}\left[ 
\vertiii{\frac{\mt{X}'\mt{W}}{T} }_{\infty}
 \le 
\mathbb{Q}(\mt{X},\mt{W},\bstar)  
\mathbb{R}(p,q,T)
\right] 
\ge
 1-8 \exp(-b\log(pq))
$$
where 
\begin{align*}
\mathbb{Q}(\mt{X},\mt{W},\bstar)
&=
8 \pi\sqrt{\frac{(b+1)}{\tilde{c}}}
\bpar{
\vertiii{\Sigma_{{X}}}
\bpar{1+ \max_{1\le i \le p}\vertii{\bstar_{:i}}_2^2 }
+ \vertiii{\Sigma_{{Y}} } }
\\
\mathbb{R}(p,q,T)& =S_{\alpha}(T)\sqrt{\frac{\log(pq)}{T} } .
\end{align*}

\end{pr}
%\edit{added sketch of proofs after each proposition per review 2's request}
%\begin{proof}[Sketch of proof for Proposition~\ref{result:rhoDB}]
%content...
%\end{proof}

\begin{rem}
Note that the free parameter $b$ serves to trade-off between the success probability on the one hand and the sample size threshold and multiplier function $\mathbb{Q}$ on the other. A large $b$
increases the success probability but worsen the sample size threshold and the multiplier function.
\end{rem}

%\comment{
%@Ambuj:
%Reviewer 1 says:
%``
%On pages 10--11, $\alpha$ is used for denoting both the $\alpha$-mixing coefficient and the minimum eigenvalue, but this is confusing. The same thing can be said about K."
%I would suggest not to modify alpha and K since they are all over the place in the proofs. Changing them might create unexpected bugs somewhere which are hard to catch all. I think the alpha's are fine since alpha with brackets denote the alpha coefficients and the one without is the lower RE coefficient. Please let me know what you think. 
%}
%
%\ambuj{I defined a new macro {\tt $\backslash$alphac} for the curvature in RE. We now use $\alphac$ for curvature and $\alpha$ for mixing. BTW, what is the problem with $K$? Are we using it in two ways? It seems to me that we pretty much use it solely in the subweibull constant sense...} 

%\subsection{High Probability Guarantee for Lower Restricted Eigenvalue}
\begin{pr}[RE, Gaussian Case] \label{result:rhoRE}
Suppose Assumptions \ref{as:stat}--\ref{as:alphaMix} hold. There exists some universal constant $c>0$, such that for sample size  \black{$T\ge \frac{42 e \log(p)}{c \min\{1, \eta^2\}}$}, we have, with probability at least $1-2\exp\left( -\frac{c}{2}T \min\{1,\eta^2 \} \right)$ that
for every vector $ \vc{v} \in \R^p $,
\begin{equation}\label{eq:gRE}
|\vc{v}'\hat{\Gamma} \vc{v} |
>\alphac \Vert \vc{v}\Vert_2^2 - 
\tau(T,  p ) 
\Vert \vc{v} \Vert_1^2  ,\\
\end{equation}
where 
\begin{align*}
\alphac &= \frac{1}{2} \lmin{\Sigma_{X}},&
\tau(T,p) &=  \alphac/\ceil{c\frac{T}{4 \log(p)} \min\{1,\eta^2\}},&\text{and} \\
\eta =&
\frac{\lmin{\Sigma_X}}{108\pi  S_{\alpha}(T)\lmax{\Sigma_X}}.
\end{align*}
\end{pr}

\begin{rem}
Note that, in Theorem \ref{result:master}, it is advantageous to have a large $ \alphac $ and a smaller $ \tau $ so that the convergence rate is fast and the initial sample threshold for the result to hold is small. The result above, therefore, clearly shows that it is advantageous to have a well-conditioned $ \Sigma_X $.
\end{rem}

%\begin{rem}
%Note that the assumption on summable $\rho$-mixing coefficients may look stringent, but it in fact is satisfied in a vast class of time series processes. For instance, \cite{tjostheim1990non} showed that for any vector autoregressive (VAR) model whose transition matrix has finite spectral radius, the time series generated from that is absolutely regular with exponentially decaying $\beta$-mixing coefficients. \\   
%\end{rem}
%
%\green{
%\begin{rem}
%\cite{chen2010nonlinearity} mentions that \cite[Lemma 3.11]{banon1977estimation} and \cite[Theorem 4.2]{bradley1985basic} 
%established that for a stationary Markov process either the
%$\rho$-mixing coefficients decay exponentially or they are identically
%equal to one. 
%\end{rem}
%}
% % % % % % % % % % % %

% % % % % % % % % % % %

\subsection{Estimation and Prediction Errors}
Substituting the RE and DB constants from Propositions~\ref{result:rhoDB}-\ref{result:rhoRE} into Theorem~\ref{result:master} immediately yields the following guarantees.
\begin{cor}[Lasso Guarantees for Gaussian Vectors under $\alpha$-Mixing]
\label{cor:gauss}
Suppose Assumptions \ref{as:stat}--\ref{as:alphaMix} hold.
Let $c,\tilde{c}$ be fixed constants and $b$ be free parameter defined as in Propositions~\ref{result:rhoDB} and \ref{result:rhoRE}. Then, for sample size
\begin{align*}
T&\ge \max\bcur{
\frac{\log(p)}{c\min\{1, \eta^2\}}
\max\bcur{
42e, 128s
}
,
\log(pq)\sqrt{\frac{b+1}{\tilde{c}}}
}
&\\
&\text{where }
\eta =
 \frac{\lmin{\Sigma_X}}{108\pi  S_{\alpha}(T)\lmax{\Sigma_X}}
\end{align*}
we have, with probability at least
$
1
-
2\exp\left( -\frac{c}{2}T \min\{1,\eta^2 \} \right)
-
8 \exp(-b\log(pq))
$,
that the lasso error bounds~\eqref{eq:l2errorBdd} and~\eqref{eq:predErrBdd} hold with
\begin{align*}
\alphac &= \frac{1}{2} \lmin{\Sigma_{X}} \\
\lambda_T&= 4 \mathbb{Q}(\mt{X},\mt{W},\bstar) \mathbb{R}(p,q,T)
\end{align*}
where
\begin{align*}
\mathbb{Q}(\mt{X},\mt{W},\bstar)
&=
8 \pi\sqrt{\frac{(b+1)}{\tilde{c}}}
\bpar{
\vertiii{\Sigma_{X}}
\bpar{1+ \max_{1\le i \le p}\vertii{\bstar_{:i}}_2^2 }
+ \vertiii{\Sigma_{Y} } } ,
\\
\mathbb{R}(p,q,T)& =S_{\alpha}(T)\sqrt{\frac{\log(pq)}{T} } .
 \end{align*}
\end{cor}

\begin{rem}
If the $ \alpha$-mixing coefficients are summable, i.e., $ S_{\alpha}(T) \le \tilde{\alpha} <\infty,\, \forall T$, then we get the usual convergence rate of $O(\sqrt{\tfrac{\log(pq)}{T}})$. Also, the threshold sample size is $O \bpar{s\log(pq)}$. This is in agreement with what  happens in the iid Gaussian case. 
When $\alpha(l)$ is not summable then both the initial sample threshold required for the guarantee to be valid as well as the rate of error decay deteriorate. The latter becomes \editt{ $O(S_{\alpha}(T)\sqrt{\tfrac{\log(pq)}{T} })$}. We see that as long as $ S_{\alpha}(T) \in o\left( \sqrt{T} \right) $, we still have consistency. In the finite order stable Gaussian VAR case considered by \cite{basu2015regularized}, the $\alpha$-mixing coefficients are geometrically decaying and hence summable (see Example~\ref{ex:GaussVAR} for details).
\end{rem}

%\todo{compare with iid}

\subsection{Examples}
We illustrate applicability of our theory developed in this section using the examples below. 
%In the following examples, we identify $ X_t := Z_t $ and $ Y_t :=Z_{t+1}$ for $ t=1,\ldots,T$. For the specific parameter matrix $ \bstar  $ in each Example below, we can verify  that Assumptions \ref{as:sparse}--\ref{assum:beta} hold (see Appendix~\ref{Apnx:Ver}) for details. Therefore, Propositions~\ref{results:REsub} and \ref{result:betaDev} and Corollary \ref{cor:subgauss} follow. Hence we have all the high probabilistic guarantees for lasso on data generated from DGM potentially involving subgaussianity, model mis-specification, and/or nonlinearity. 

\begin{exmp}[Gaussian VAR]\label{ex:GaussVAR} Transition matrix estimation in sparse stable VAR models has
been considered by several authors in recent years \citep{davis2016sparse,han2013transition,song2011large}.
The lasso estimator is a natural choice for the problem.
 
Formally a finite order Gaussian VAR($ d $) process is defined as follows.
Consider a sequence of serially ordered random vectors $(Z_t)_{t=1}^{T+d}$, ${Z_t}\in \R^p$ that admits the following auto-regressive representation:
\begin{align}\label{eq:VAR(1)}
{Z_t} = \mt{A}_1 {Z}_{t-1}+ \dots +  \mt{A}_d {Z}_{t-d} + {\mathcal{E}}_t
\end{align}
where each $\mt{A}_k, k=1, \dots, d$ is  a sparse non-stochastic coefficient matrix in $\R^{p \times p}$ and innovations ${\mathcal{E} }_t$ are $p$-dimensional random vectors from $\mathcal{N}( \vc{0}, \Sigma_{\epsilon})$ with $\lmin{\Sigma_{\epsilon}}>0$ and $\lmax{\Sigma_{\epsilon}}< \infty$. 

Assume that the VAR($ d $) process is \emph{stable}; i.e. $ \mathrm{det}\bpar{\mt{I}_{p \times p}-\sum_{k=1}^{d}\mt{A}_k z^k} \neq 0,\, \forall \verti{z} \le 1 $. 
%Note that every VAR(d) process $ (Z_t) $ has an equivalent augmented VAR($ 1 $) representation $ (\tilde{Z}_t) $ where $\tilde{Z}_t \in \R^{pd}$ (see Appendix \ref{veri:VAR}). Further, if the process $ (X_t) $ is stable and the coefficient matrices $ A_1,\cdots,A_d $ are sparse then the corresponding VAR(1) process $ (\tilde{X}_t) $ is also stable and has sparse coefficient matrix.
Now, we identify $X_t:=({Z}'_t,\cdots,{Z}'_{t-d+1})'$ and $Y_t :=Z_{t+d}$ for $t = 1,\ldots,T$. 

We can verify (see Appendix \ref{veri:VAR} for details) that Assumptions \ref{as:spars}--\ref{as:alphaMix} hold. Note that $\bstar = (\mt{A}_1,\ldots,\mt{A}_d)' \in \R^{dp \times p}$.
As a result, Propositions  \ref{result:rhoDB} and \ref{result:rhoRE}, and thus Corollary \ref{cor:gauss} follow   and hence we have all the high probabilistic guarantees for lasso on Example \ref{ex:GaussVAR}. This shows that our theory covers the stable Gaussian VAR models for which \cite{basu2015regularized} provided lasso errors bounds. \\

We state the following convenient fact because it allows us to study any finite order VAR model by considering its equivalent VAR($ 1 $) representation. See Appendix \ref*{veri:VAR} for details.

\begin{fact}
Every VAR($d$) process can be written in VAR($ 1 $) form (see e.g. \cite[Ch 2.1]{lutkepohl2005new}).
\end{fact}

Therefore, without loss of generality, we can consider VAR($ 1 $) model in the ensuing Examples. 
%Assume that it is a (1) \textbf{Gaussian process} and (2) can be represented as a linear regression on its own lags. \\

%
%\begin{figure}[h]
%\centering
%\includegraphics[width=0.7\linewidth]{plots/GaussianVAR}
%\caption{$ l_2 $ estimation error of transition matrix against rescaled dimenions. }
%\label{fig:GaussianVAR}
%\end{figure}
%
%From Fig. \ref{fig:GaussianVAR}, we see it fits with our theory well. \\
\end{exmp}

\begin{exmp}[Gaussian VAR with Omitted Variable]\label{ex:misVAR}
\edit{We  study lasso estimation for a VAR(1) process when there are endogenous variables omitted. This arises naturally when the underlying DGM is high-dimensional but not all variables are available (e.g., it is impossible to observe them or perhaps very costly to measure them) to the researcher to perform estimation and prediction. Such a situation can also arise when the researcher mis-specifies the scope of the model.}

\edit{Notice that the system of the retained set of variables is no longer a finite order VAR (and thus non-Markovian). As we describe below, the target of estimation is still the best linear predictor (in the least squares sense) of the future given the past.} There is model mis-specification and this example also serves to illustrate that our theory is applicable to models beyond the finite order VAR setting.

Consider a VAR(1) process $ (Z_t, \Xi_t)_{t=1}^{T+1} $ such that each vector in the sequence is generated by the recursion below:
$$
(Z_t; \Xi_t) = \mt{A} (Z_{t-1}; \Xi_{t-1}) + (\mathcal{E} _{Z, t-1}; \mathcal{E} _{\Xi, t-1})
$$
where $Z_t \in \R^{p} $,  $\Xi_t \in \R $, $ \mathcal{E} _{Z, t} \in \R^{p}  $, and $\mathcal{E} _{\Xi, t} \in \R $ are partitions of the random vectors $ (Z_t, \Xi_t) $ and $ \mathcal{E}_t $ into $ p $ and $ 1 $ variables. Also,
$$
\mt{A}:= 
\bbra{
\begin{array}{cc}
\mt{A}_{ZZ} & \mt{A}_{Z\Xi} \\ 
\mt{A}_{\Xi Z} & \mt{A}_{\Xi \Xi}
\end{array}  
}
$$
is the coefficient matrix of the VAR(1) process with $\mt{A}_{Z \Xi }   $ $ 1 $-sparse, $ \mt{A}_{Z Z } $  $ p $-sparse and $ r(\mt{A})<1 $. $ \mathcal{E}_t := (\mathcal{E} _{X, t-1}; \mathcal{E} _{Z, t-1}) $ for $ t=1,\ldots,T+1$ are iid draws from a Gaussian white noise process. 

We are interested in the best (in the least squares sense) $ 1 $-lag predictor of $Z_t$ as a function of $ Z_{t-1}$. Recall that
$$
\bstar := \argmin_{\mt{B} \in \R ^{p\times p}} \E \bpar{  
\vertii{
Z_t - \mt{B}'Z_{t-1}
}^2_2
}
$$
Note that $Z_t$ is not necessarily a finite order VAR process. Now,  set $ X_t := Z_t $ and $ Y_t :=Z_{t+1}$ for $ t=1,\ldots,T$. It can be shown that $ (\bstar)' =\mt{A}_{Z Z }+\mt{A}_{Z \Xi } \Sigma_{\Xi Z }(0)(\Sigma_Z)^{-1} $.
We can verify that Assumptions \ref{as:spars}--\ref{as:alphaMix} hold. See Appendix \ref{veri:misVAR} for details.
As a result,  Propositions  \ref{result:rhoDB} and \ref{result:rhoRE}, and thus Corollary \ref{cor:gauss} follow   and hence we have all the high probabilistic guarantees for lasso on this non-Markovian example.

%\begin{figure}
%\centering
%\includegraphics[width=0.7\linewidth]{plots/GaussainMisspecifiedVAR}
%\caption{$ l_2 $ estimation error of transition matrix against rescaled dimenions.}
%\label{fig:GaussianMisspecifiedVAR}
%\end{figure}
%Again, we see from Fig. \ref{fig:GaussianMisspecifiedVAR} that our theory is verified. \\ \\
\end{exmp}

%% file: subweibull-beta.tex
%!TEX root = aos-mixing.tex

\section{Subweibull Random Vectors under $\beta$-Mixing}\label{section:subweibull}
%
%\red{To-Dos
%\begin{enumerate}
%\item
%Move the subweibull definitions to preliminaries? But these are not standard known definitions...
%\item 
%Shall we include the subgaussian definitions as motivations for the subweibull ones?
%\item
%Remove the subgaussian concentration lemma from the subgaussian section to be consistent with the subweibull one? We did not include a concentration in the main text in subweibull section. 
%\end{enumerate}
%}

Existing analyses of lasso mostly assume data have subgaussian or subexponential tails. These assumptions ensure that the moment generating function exists, at least for some values of the free parameter.
Non-existence of the moment generating function is often taken as the definition of having a heavy tail~\citep{foss2011introduction}. 
We now introduce a family of random variables that subsumes subgaussian and subexponential random variables. In addition, it includes some heavy tailed distributions. 

\subsection{Subweibull Random Variables and Vectors}\label{section:subweibulldef}
%The subgaussian and subexponential random variables are characterized by the behavior of their tails. 
Among the several equivalent definitions of the subgaussian and subexponential random variables, we recall the ones that are based on the growth behavior of moments. Recall that a subgaussian (resp. subexponential) random variable $X$ can be defined as one for which $\E(|X|^p)^{1/p} \le K\sqrt{p}, \,\forall p\ge 1$ for some constant $K$ (resp. $ \E(|X|^p)^{1/p} \le Kp, \,\forall p\ge 1 $).
A natural generalization of these definitions that allows for heavier tails is as follows. Fix some $\gamma > 0$, and require
\begin{align*}
\vertii{X}_p:= (\E {\verti{X}^{}}^p)^{1/p} \le K p^{1/\gamma}, \;\forall p\ge 1 \wedge \gamma
\end{align*} 

%Each member of the family is characterized by a parameter $ \gamma_2 $, denoted $ \mathrm{SW}(\gamma_2) $. 
There are a few different equivalent ways to imposing the condition above including a tail condition that says that the tail is no heavier than that of a Weibull random variable with parameter $ \gamma$. That is the reason why we call this family ``subweibull$ (\gamma)"$.

%\red{remove the definition below (define a SW rv after the properties}
%\begin{defn}\label{def:subWeibull}(subweibull($ \gamma_2 $) random variable).
%A random variable $ X $ is said to follow a subweibull distribution with parameter $ \gamma_2 $, denoted $ \mathrm{SW}(\gamma_2) $ if the moments of $ X $ satisfies, for some constant  $K>0  $,
%\begin{align*}
% \vertii{X}_p:= (X^p)^{1/p} \le K p^{1/\gamma_2}, \,\forall p\ge 1 
%\end{align*}
%\end{defn}

\begin{lm}\label{lem: subWei}
(Subweibull properties)
Let $ X $ be a random variable. Then the following statements are equivalent for every $ \gamma > 0  $. The constants $ K_1, K_2, K_3 $ differ from each other at most by a constant depending only on $\gamma$. 
\begin{enumerate}
\item The tails of $ X $ satisfies
\begin{align*}
\prob\bpar{\verti{X} >t } \le 2 \exp\bcur{-(t/K_1)^{\gamma}},\; \forall t\ge 0 .
\end{align*} 
\item The moments of $ X $ satisfy, 
\begin{align*}
\vertii{X}_p:= (\E {\verti{X}^{}}^p)^{1/p} \le K_2 p^{1/\gamma}, \;\forall p\ge 1\wedge \gamma.
\end{align*}
\item The moment generating function of $ \verti{X}^{\gamma} $ is finite at some point; namely
\begin{align*}
\E \bbra{\exp\bpar{\verti{X}/K_3}^{\gamma} } \le 2 .
\end{align*}
\end{enumerate}

\end{lm}

\begin{rem}
A similar tail condition is called ``Condition C0'' by \cite{tao2013random}. However, to the best of our knowledge, this family has not been systematically introduced.
The equivalence above is related to the theory of Orlicz spaces (see, for example, Lemma 3.1 in the lecture notes of \cite{pisier2016subgaussian}).
\end{rem}

\begin{defn}(Subweibull($ \gamma $) Random Variable and Norm).
A random variable $ X $ that satisfies any property in Lemma~\ref{lem: subWei} is called a subweibull($ \gamma $) random variable. The subweibull($ \gamma $) norm associated with  $ X $, denoted $ \vertii{X}_{\psi_\gamma} $, is defined to be the smallest constant such that the moment condition in definition Lemma~\ref{lem: subWei} holds. In other words, for every $ \gamma >0 $,
\begin{align*}
\vertii{X}_{\psi_\gamma}:= \sup_{p\ge 1 }(\E {\verti{X}^{}}^p)^{1/p}  p^{-1/\gamma} .
\end{align*}
\end{defn}

It is easy to see that $ \vertii{\cdot}_{\psi_\gamma} $, being a pointwise supremum of norms, is indeed a norm on the space of subweibull($\gamma$) random variables.

%Let's check the axioms:
%\begin{itemize}
%\item $ \vertii{0}_{\gamma_2} = 0 $
%\item for any constant $ c \in \R $, $ \vertii{cX}_{\gamma_2} = c\vertii{\cdot}_{\gamma_2} $ because $  \vertii{X}_p $ is homogeneous.
%\item $ \vertii{\cdot}_{\gamma_2}  $ is convex because $  \vertii{\cdot}_p $ is convex and  point-wise maximum preserves convexity. 
%\end{itemize}

\begin{rem}
It is common in the literature (see, for example \cite{foss2011introduction}) to call a random variable \emph{heavy-tailed} if its tail decays slower than that of an exponential random variable. This way of distinguishing between light and heavy tails is natural because the moment generating function for a heavy-tailed random variable thus defined fails to exist at any point. Note that, under such a definition, subweibull($ \gamma$) random variables with $ \gamma <1 $ include heavy-tailed random variables.
\end{rem}

In our theoretical analysis, we will often be dealing with squares of random variables. The next lemma tells us what happens to the subweibull parameter $\gamma$ and the associated constant, under squaring.

\begin{lm}\label{lemma:gammaNormOfSquares}
For any $ \gamma \in (0, \infty) $,
if a random variable $ X $ is subweibull($ 2\gamma$) then $ X^2 $ is subweibull($ \gamma$). Moreover, 
%for some constant $ \omega(\gamma_2) >0 $ that depends solely on $ \gamma_2 $, 
\begin{align*}
\norm[X^2]{\psi_\gamma} \le 2^{1/\gamma}\norm[X]{\psi_{2 \gamma}}^2 .
\end{align*}
\end{lm}

%
%\begin{defn}\label{def:tailRV}
%For a constant $ \gamma_2 \in (0, \infty] $, a random variable $ X \in \R$, is said to follow a subweibull($ \gamma_2 $),  if 
%\begin{align*}
%\norm[X]{\gamma_2} <\infty
%\end{align*}
%\end{defn}

We now define the subweibull norm of a random vector to capture dependence among its coordinates. It is defined using one dimensional projections of the random vector in the same way as we define
subgaussian and subexponential norms of random vectors. 

\begin{defn}\label{def:vector}
Let $ \gamma \in (0, \infty)$. A random  vector $ X \in \R^p $ is said to be a subweibull($ \gamma$) random vector if all of its one dimensional projections are subweibull($ \gamma$) random variables. We define the
subweibull($\gamma$) norm of a random vector as,
\begin{align*}
\norm[X]{\psi_\gamma} := \sup_{v \in S^{p-1}} \norm[v'X]{\psi_\gamma}
\end{align*}
where $ S^{p-1} $ is the unit sphere in $ \R^p $.
\end{defn}

%Beyond sub-subexpoentiality, there is no such analogue. We would like to extend the idea beyond sub-exponentiality. 

%Foe illustration, take the multivariate normal $ X=(X_1,\cdots, X_p)\in \R^p $ as an example; it is sub-Gaussian. If $ X_1=...=X_p $, then the sub-Gaussian norm $ \snorm{X} \in \mathcal{O}(p)$; on the other hand, if $ X \sim \mathbb{N}(0, I_p) $, then  $ \snorm{X} \in \mathcal{O}(1)$. Similarly, when $ X\in \R^p $, $\knorm{X} $ measures departure from isotropy.
%Define the $ \gamma_2 $ norm of $ X $ as
%$$  
%\vertii{X}_{\gamma_2}:=\inf \{ b\ge 0 : \text{condition } \eqref{def: tailCond} \text{ holds for all }t>0\} .
%$$ 
%Then, $ K(\gamma_2, X) $ is the biggest $ \gamma_2 $ norm among all the one dimensional projections of $ X $. Precisely, paralleling the definition of multivariate sub-Gaussian distributions, we define random vectors that satisfies Definition \ref{def: tailCond} as follows,

Having introduced the subweibull family, we present the assumptions required for the lasso guarantees. In proving our results, we need measures that control the amount of dependence in the observations across time as well as within a given time period. 

\begin{assump}\label{assump:betamixing}
The process $ (X_t,Y_t) $ is geometrically $ \beta$-mixing; i.e., there exist constants $ c>0 $ and $ \gamma_1>0 $ such that
\begin{align*}
\beta(n) \le 2\exp(-c\cdot n^{\gamma_1}),\; \forall n\in \mathbb{N}.
\end{align*}
\end{assump}

%\comment{Replaced $ < $ with  $ \le $ per review 1's point 5. Reviewer said add a comment, but I think a comment is not needed?}\\
%\ambuj{Agreed. Comment is not needed.}

\begin{assump}\label{assump:subW}
Each random vector in the sequences $ (X_t) $ and $ (Y_t) $ follows a  subweibull($ \gamma_2 $) distribution with $\norm[X_t]{\psi_{\gamma_2}} \le K_X$, $\norm[Y_t]{\psi_{\gamma_2}} \le K_Y$ for $t=1,\cdots, T$.
\end{assump}

Finally, we make an joint assumption on the allowed pairs $\gamma_1, \gamma_2$. 
%\red{TODO: this assumption needs to be discussed. It seems a bit odd to me though. E.g., in the independent across time case $\gamma_1 = \infty$,
%so we demand $\gamma_2 < 2$, i.e., tails be heavier than subgaussian. Also, suppose the random vectors are almost surely bounded which means $\gamma_2 = \infty$. Then we demand $\gamma_1 < 1$, i.e. independence over time isn't allowed! Doesn't
%this seem a bit counter-intuitive. In other words, decreasing dependence across time or having lighter tails can violate this condition!}

\begin{assump} \label{assump:gammaSmall}
Assume $ \gamma<1 $ where
\begin{align*}
 \gamma &:= \bpar{\frac{1}{\gamma_1}+\frac{2}{\gamma_2}}^{-1} .
 \end{align*}
\end{assump}

\begin{rem}\label{remark:gamma assumption}
%TODO: make more concise
Note that  the parameters $ \gamma_1 $ and $ \gamma_2 $ defines a difficulty landscape with smaller values of $\gamma_1, \gamma_2$ corresponding to harder problems. The ``easy case" where $ \gamma_1 \ge 1 $ and $ \gamma_2 \ge 2$ are already addressed in the literature (see, e.g.,~\cite{wong2016lasso}). This paper serves to provide theoretical guarantees for the difficult scenario when the tail probability decays slowly ($ \gamma_2 <2 $) and/or data exhibit strong temporal dependence ($ \gamma_1 < 1 $) and hence extends the literature to the entire spectrum of possibilities, i.e., all positive values of $ \gamma_1 $ and $ \gamma_2 $.
%\red{Note that this Assumption rules out the cases where (1) data are independent ($ \gamma_1 = \infty $) and have lighter than subgaussian tails ($ \gamma_2 \ge 2 $), or (2) we have almost surely bounded data $ \gamma_2 =\infty $ and temporal dependence decaying too fast $ \gamma_1 \ge 1 $. However these cases are already addressed in the literature (e.g. in~\cite{wong2017lasso}). This paper serves to provide theoretical guarantees for the difficult scenario when the tail probability decays slowly ($ \gamma_2 <2 $) and data exhibit temporal dependence ($ \gamma_1 <\infty $) and hence extends the literature to all spectrum of positive $ \gamma_1 $ and $ \gamma_2 $.}
\end{rem}

Now, we are ready to provide high probability guarantees for the deviation bound and restricted eigenvalue conditions.

\begin{pr}[Deviation Bound, $ \beta $-Mixing Subweibull Case]\label{Lemma:SubweibullDevBound}
Suppose Assumptions \ref{as:spars}-\ref{as:0mean} and \ref{assump:betamixing}-\ref{assump:gammaSmall} hold.  Let $c'>0$ be a universal constant and let $K$ be defined as
\begin{align*}
 K := 2^{2/\gamma_2}\bpar{{K_{Y}} + {K_{X}}\bpar{1+\vertiii{\bstar} } }^2 .
 \end{align*}
Then with sample size $T \ge C_1 (\log(pq))^{\tfrac{2}{\gamma} -1 }$,
we have 
\begin{align*}
\prob \bpar{ \frac{1}{T} \vertiii{\mt{X}'\mt{W}}_{\infty} > C_2 K \sqrt{ \frac{\log (pq) }{T} }  }
&\le 2\exp(-c' \log(pq))
 \end{align*}
where the constants $C_1, C_2$ depend only on $c'$ and the parameters $\gamma_1,\gamma_2, c$ appearing in Assumptions~\ref{assump:betamixing} and~\ref{assump:subW}.
\end{pr}

\begin{pr}[RE, $ \beta $-Mixing Subweibull Case]\label{Lemma:SubweibullRE}
Suppose Assumptions \ref{as:spars}-\ref{as:0mean} and \ref{assump:betamixing}-\ref{assump:gammaSmall} hold.
Let 
\begin{align*}
K: =2^{2/\gamma_2}K_X^2 .
\end{align*}
Then for sample size
\begin{align*}
T  \ge
\max \bcur{
\frac{54 K \bpar{2  C_1\log(p)}^{1/\gamma}}{\lmin{\Sigma_X}}
,\,
\left(\frac{54 K}{\lmin{\Sigma_X}}\right)^{\frac{2-\gamma}{1-\gamma}} \left(\frac{C_2}{C_1}\right)^{\frac{1}{1-\gamma}}
}
\end{align*}
we have with probability at least
\begin{align*}
1-  2T \exp  \bcur{   -\tilde{c} T^\gamma}, \text{ where } \tilde{c} = \frac{\bpar{\lmin{\Sigma_X}}^\gamma}{(54 K)^\gamma 2C_1},
\end{align*}
that for all $v \in \R^p$,
\begin{align*}
\frac{1}{T}\vertii{\mt{X}v}^2_2 &\ge  
\alphac \vertii{v}^2_2   -\tau\vertii{v}^2_1 .
\end{align*}
where $\alphac = \half \lmin{\Sigma_X}$ and $\tau = \frac{\alphac}{2\tilde{c}}  \cdot\bpar{ \frac{ \log(p)}{T^\gamma }}$.
Note that the constants $C_1, C_2$ depend only on the parameters $\gamma_1,\gamma_2, c$ appearing in Assumptions~\ref{assump:betamixing} and~\ref{assump:subW}.
\end{pr}
%\red{put tgt}
\subsection{Estimation and Prediction Errors}

Substituting the RE and DB constants from Propositions~\ref{Lemma:SubweibullDevBound}-\ref{Lemma:SubweibullRE} into Theorem~\ref{result:master} immediately yields the following guarantee.
\begin{cor}[Lasso Guarantees for Subweibull Vectors under $\beta$-Mixing]
\label{cor:Subweibull}
Suppose Assumptions \ref{as:spars}-\ref{as:0mean} and \ref{assump:betamixing}-\ref{assump:gammaSmall} hold.  Let $c', C_1, C_2, \tilde{c}$ be constants as defined in Propositions~\ref{Lemma:SubweibullDevBound}-\ref{Lemma:SubweibullRE}, and let $ K := 2^{2/\gamma_2}\bpar{{K_{Y}} + {K_{X}}\bpar{1+\vertiii{\bstar} } }^2 $.

Then for sample size

 \begin{align*}
 T  \ge&
 \max \left\{C_1 (\log(pq))^{\tfrac{2}{\gamma} -1 },\,\right. \\
&\quad\quad\quad\left. \frac{54 K \bbra{2\max\{8s/\tilde{c},C_1\} \log(p)}^{1/\gamma}}{\lmin{\Sigma_X}} 
 ,
 \left(\frac{54 K}{\lmin{\Sigma_X}}\right)^{\frac{2-\gamma}{1-\gamma}} \left(\frac{C_2}{C_1}\right)^{\frac{1}{1-\gamma}}
 \right\}
 \end{align*}

 we have with probability at least
 $$
 1
 -
2T \exp  \bcur{   -\tilde{c} T^\gamma}
 -
2\exp(-c' \log(pq))
 $$
that the lasso error bounds~\eqref{eq:l2errorBdd} and~\eqref{eq:predErrBdd} hold with
 \begin{align*}
 \alphac &= \frac{1}{2} \lmin{\Sigma_{X}} \\
 \lambda_T&= 4 \mathbb{Q}(\mt{X},\mt{W},\bstar) \mathbb{R}(p,q,T)
 \end{align*}
 where
 \begin{align*}
 \mathbb{Q}(\mt{X},\mt{W},\bstar)
 &=
 C_2K ,
 \\
 \mathbb{R}(p,q,T)& =\sqrt{\frac{\log(pq)}{T} } .
  \end{align*}
\end{cor}

\begin{rem}
The impact of mixing behavior is limited to the initial sample size and the probability with which the error bounds hold. The parameter error bound itself resembles the bounds obtained in the iid case
but with an additional multiplicative factor that depends on the ``effective condition number" $K/\lmin{\Sigma_{X}}$. 
\end{rem}

%\red{TODO: examples and verification all need to be checked}
\subsection{Examples}
We explore applicability of our theory in Section~\ref{section:subweibull} beyond just linear Gaussian processes using the examples
 below. Together, these demonstrate that the high probabilistic guarantees for lasso cover cases of heavy tailed subweibull data, presence of model mis-specification, and/or nonlinearity.
\begin{exmp}[Subweibull VAR]\label{examp:subWeilVAR}

We study a generalization of the VAR, one that has subweibull($ \gamma_2 $) realizations. Consider a VAR($ 1 $) model defined as in Example \ref{ex:GaussVAR} except that we replace the Gaussian white noise innovations with iid random vectors from some subweibull($ \gamma_2 $) distribution with a non-singular covariance matrix $ \Sigma_{\epsilon} $. Now, consider a sequence $ (Z_t)_t $ generated according to the model. Then, each $ Z_t $ will be a mean zero subweibull random vector.

% Assume $\lmin{\Sigma_{\epsilon}}>0$ and $\lmax{\Sigma_{\epsilon}}< \infty$. 
%Note that every VAR(d) process $ (Z_t) $ has an equivalent augmented VAR($ 1 $) representation $ (\tilde{Z}_t) $ where $\tilde{Z}_t \in \R^{pd}$ (see Appendix \ref{veri:VAR}). Further, if the process $ (X_t) $ is stable and the coefficient matrices $ A_1,\cdots,A_d $ are sparse then the corresponding VAR(1) process $ (\tilde{X}_t) $ is also stable and has sparse coefficient matrix.
Now, we identify $X_t:=({Z}'_t,\cdots,{Z}'_{t-d+1})'$ and $Y_t :=Z_{t+d}$ for $t = 1,\ldots,T$.
Assuming that $\mt{A}_i$'s are sparse, 
\edit{
$ r(\mt{A}) <1 $,
% and $(Z_t)_t$ is stable
}
we can verify (see Appendix~\ref{veri:VAR} for details) that Assumptions \ref{as:spars}-\ref{as:0mean} and \ref{assump:betamixing}-\ref{assump:gammaSmall} hold. Note that $\bstar = (\mt{A}_1,\ldots,\mt{A}_d)' \in \R^{dp \times p}$.
As a result, Propositions~\ref{Lemma:SubweibullDevBound} and \ref{Lemma:SubweibullRE} follow   and hence we have all the high probability guarantees for lasso on Example~\ref{examp:subWeilVAR}.
% \red{The next sentence is a repetition of the Gaussian VAR example. Delete?}
This shows that our theory covers DGMs beyond just the stable Gaussian processes.

%
%Consider a multivariate autoregressive process $ (Z_t) $ that admits the following representation
%\begin{align*}
%fill in...
%\end{align*}
%Assume that the operator norm of the transition matrix is finite; i.e. $ \vertiii{A} <1 $.

\end{exmp}

\begin{exmp}[VAR with Subweibull Innovations and Omitted Variable]\label{ex:sGmisVAR}

Using the same setup as in Example \ref{ex:misVAR} except that we replace the Gaussian white noise innovations with iid random vectors from some subweibull($ \gamma_2 $) distribution with a non-singular covariance matrix $ \Sigma_{\epsilon} $. Now, consider a sequence $ (Z_t)_t $ generated according to the model. Then, each $ Z_t $ will be a mean zero subweibull random vector.

%the Gaussian white noise innovations with sub-Gaussian ones  and assume that $\vertiii{\mt{A}} < 1$. 

%For example, we can take iid random vectors from the uniform distribution; i.e. $\forall t,\,\mathcal{E}_t \overset{iid}{\sim} \U{\bbra{-\sqrt{3}, \sqrt{3}}^{p+1}} $. These $\mathcal{E}_t$ will be independent centered isotropic sub-Gaussian random vectors,
%giving us we a VAR($ 1 $) model with sub-Gaussian innovations.  Now, consider a sequence $ (Z_t,\mathcal{E}_t)_{t=1}^{T+1} $ generated according to the model. Then, each $ (Z_t,\mathcal{E}_t) $ will be a mean zero sub-Gaussian random vector.

 Now,  set $ X_t := Z_t $ and $ Y_t :=Z_{t+1}$ for $ t=1,\ldots,T$. Assume \edit{$ r(\mt{A})<1 $}. It can be shown that $ (\bstar)' =\mt{A}_{Z Z }+\mt{A}_{Z \Xi } \Sigma_{\Xi Z }(0)(\Sigma_Z )^{-1} $.
 We can verify  that Assumptions \ref{as:sparse}-\ref{as:0mean} and \ref{assump:betamixing}-\ref{assump:gammaSmall} hold. See Appendix \ref{veri:misVAR} for details. Therefore, Propositions~\ref{Lemma:SubweibullDevBound} and~\ref{Lemma:SubweibullRE} and thus Corollary~\ref{cor:Subweibull} follow and hence we have all the high probabilistic guarantees for subweibull random vectors from a non-Markovian model. 
%Results (Fig. \ref{fig:Sub-gaussianMisspecifiedVAR}) look reasonable. 
%
%\begin{figure}
%\centering
%\includegraphics[width=0.7\linewidth]{plots/Sub-gaussainMisspecifiedVAR}
%\caption{$ l_2 $ estimation error of transition matrix against rescaled dimenions.}
%\label{fig:Sub-gaussianMisspecifiedVAR}
%\end{figure}

\end{exmp}

\begin{exmp}[Multivariate ARCH]\label{ex:ARCH}
%\green{Isolate the result that the best linear predictor is the submatrix of the transition matrix?}
We  explore the generality of our theory by considering a multivariate nonlinear time series model with subweibull innovations. 
%\red{Please check the following added sentence}
A popular nonlinear multivariate time series model in econometrics and finance is the vector autoregressive conditionally heteroscedastic (ARCH) model. 
We choose the following specific ARCH model just for convenient validation of the geometric $ \beta $-mixing property of the process; it may potentially be applicable to a larger class of multivariate ARCH models. 
%A popular nonlinear multivariate time series model in econometrics and finance is the vector autoregressive conditionally heteroscedastic (ARCH) model. Consider the following specific example of  vector ARCH model.

 Let $ (Z_t)_{t=1}^{T+1} $ be random vectors  defined by the following recursion, for any constants $ c>0 $, $ m\in (0,1) $, $a > 0$, and $ \mt{A} $ sparse
with \edit{$r(\mt{A}) < 1$}:
\begin{align} 
\begin{split} \label{eq:ARCH}
Z_t&= \mt{A} Z_{t-1}+\Sigma(Z_{t-1}) \mathcal{E}_t \\
\Sigma(\vc{z}) &:= c \cdot \clip{\vertii{\vc{z}}^{m}}{a}{b}\mt{I}_{p \times p}
\end{split}
\end{align}
where $ \mathcal{E}_t    $ are iid random vectors from some subweibull($ \gamma_2 $) distribution with a non-singular covariance matrix $ \Sigma_{\epsilon} $, and $\clip{x}{a}{b}$ clips the argument $x$ to stay in the interval $[a,b]$;
\edit{
i.e., 
$$
\clip{x}{a}{b}  = 
\begin{cases} 
b &\mbox{if } x \ge 0 \\ 
x & \mbox{if } a<x<b \\ 
a & \mbox{otherwise. } 
\end{cases}
$$}

Consequently, each $ Z_t $ will be a mean zero subweibull random vector. Note that $ \Theta^\ast = \mt{A}' $, the transpose of the coefficient matrix $ \mt{A} $ here.
%We can take innovations to be iid random vectors from uniform distribution; i.e. $\mathcal{E}_t \overset{iid}{\sim} \U{\bbra{-\sqrt{3},\sqrt{3}}^p}$ for every $ t=1,\cdots, T+1 $. Then, each $\mathcal{E}_t$ will be a mean zero isotropic sub-Gaussian random vector,
%giving us we an ARCH model with sub-Gaussian innovations. Also, each $ Z_t $ will be a mean zero sub-Gaussian random vector.

% In particular, note that $ \bstar $ is precisely the transpose $\mt{A}'$ of the coefficient matrix $ \mt{A} $ here, thus sparsity is built in by construction. 
 Now,  set $ X_t := Z_t $ and $ Y_t =Z_{t+1}$ for $ t=1,\ldots,T$.  We can verify (see Appendix \ref{veri:ARCH} for details) that Assumptions \ref{as:sparse}-\ref{as:0mean} and \ref{assump:betamixing}-\ref{assump:gammaSmall} hold. Therefore, Propositions~\ref{Lemma:SubweibullDevBound} and~\ref{Lemma:SubweibullRE}, and thus Corollary~\ref{cor:Subweibull} follow and hence we have all the high probabilistic guarantees for lasso for a nonlinear models with subweibull innovations. \edit{Our example is admittedly contrived, but we hope that our techniques and results will allow other researchers to consider more compelling non-linear models.}

%\begin{figure}[H]  \label{fig:subG}
%\centering
%\subfigure[VAR with Sub-Gaussian Innovations]{
%\includegraphics[width=60mm]{plots/Sub-gaussianVAR}
%\label{fig:subfigure1}}
%\quad
%\subfigure[Sub-Gaussian VAR with Omitted Variable]{
%\includegraphics[width=60mm]{plots/Sub-gaussainMisspecifiedVAR}
%\label{fig:subfigure2}}
%\subfigure[Multivariate ARCH]{
%\includegraphics[width=60mm]{plots/MultivariateARCH}
%\label{fig:subfigure3}}
%\quad
%%
%\caption{lasso Error Bounds for Examples \ref{ex:sGVAR}, \ref{ex:sGmisVAR} and \ref{ex:ARCH}.}
%\label{fig:figure}
%\end{figure}
%
%Plots in Figure~\ref{fig:subG} show the lasso errors $ \vertii{\bstar -\bhat} $ on Examples \ref{ex:sGVAR}, ~\ref{ex:sGmisVAR} and~\ref{ex:ARCH} with sparsity $ s \approx p $. In all cases, we plotted the error versus the rescaled sample size $ T/s\log(p) $. The curves of different dimensions $ p $ align and their shapes are in agreement with the scaling in Corollary~\ref{cor:Subweibull}. It verifies our results in 　Section~\ref{subsect:subgauss}.
%\begin{figure}
%\centering
%\includegraphics[width=0.7\linewidth]{plots/MultivariateARCH}
%\caption{$ l_2 $ estimation error of transition matrix against rescaled dimenions.}
%\label{fig:MultivariateARCH}
%\end{figure}
%
%We see from Fig. \ref{fig:MultivariateARCH} that the lasso estimates behave exactly as theory predicts even in nonlinear model!
\end{exmp}

%% file: simulations.tex
%!TEX root = aos-mixing.tex

\section{Simulations}\label{sect:sim}

\edit{In this section we report simulation results to study the effect of heavy tails and temporal dependence on the estimation error of lasso.}

\subsection{\edit{Effect of Heavy-Tailedness}}\label{experiment:longtailed}
\edit{We conducted a simulation experiment to investigate the effect of heavy-tailedness via a subweibull VAR (Example~\ref{examp:subWeilVAR}).
Consider a standard VAR Model
$$
X_{t+1} = \mt{A}X_t + c\epsilon_t
$$
where the underlying parameter matrix $\mt{A}$ is an $s$-sparse $p \times p$ matrix with spectral radius $c$, $X_t$ is a $p \times 1$ vector, and $\epsilon_t$ is a subweibull random vector where each entry is iid Weibull random variable with shape parameter $\alpha$ and scale parameter 1. Moreover, $\epsilon_t$ is independent across time. For the simulation, $\mt{A}$ is generated by first randomly choosing $s$ positions with non-zero entries and then sampling each non-zero entry iid from Uniform(0,1). Finally, $\mt{A}$ is rescaled so that its spectral radius is $c = 0.5$. Now, let $s = \sqrt{p}$, $p \in \{50,100,150,200\}$, $\alpha \in \{0.4,0.5,1.0,1.9\}$, and $T = m \times s\log(p)$ where multiples $m \in \{1, 3, 5, 7, 9, 11, 13, 15, 17, 19\}$. The average estimation error (Frobenius norm of the difference between true parameter matrix $\mt{A}$ and its estimated counterpart $\hat{\mt{A}}$) over 10 repetitions plotted against $\sqrt{\frac{s \log(p)}{T}}$ is shown in Figure~\ref{fig:shape}.}

\begin{center}
\begin{figure}[h]\label{fig:shape}
\includegraphics[width=0.8\textwidth]{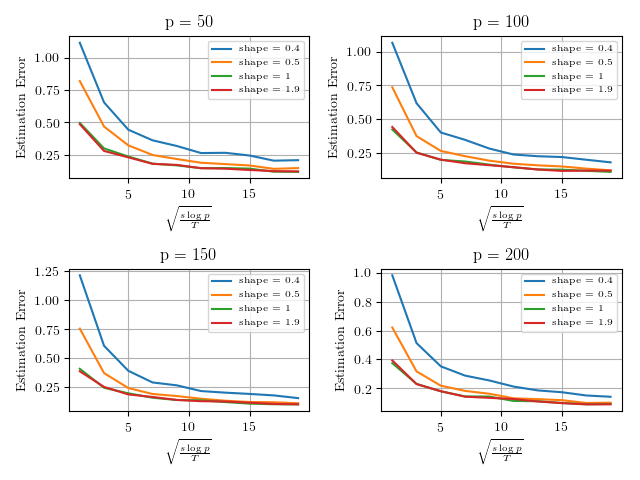}
\caption{Effect of Heavy Tails on Lasso Estimation Error. A smaller shape parameter corresponds to heavier tails of the noise.}
\end{figure}
\end{center}

\edit{The relation between shape parameter and the estimation error is quite clear. Smaller shape parameter means a heavier tail, resulting in larger estimation error, which is indeed what we observe here. The unequal spacing in the choice of shape parameter is due to the fact that the relation of shape parameter and estimation error is highly non-linear, and using equal spacing would make the differences between plots less easy to see.}

%\comment{@Ambuj: should we include the following sentences? That might be confusing to the reviewers as we do not explicitly give mathematical characterization of dependence on mixing-coefficients per se and might seem like our work is incomplete?} For example, $\alpha = 0.2$ incurs 10 time estimation error compared to $\alpha = 0.5$ and there is virtually no difference for $\alpha > 0.7$. So for example if we use shape $\in \{0.2,0.6,1,1.4\} $, then we will basically see a curve that is way above and three curves that are "grouped together", but if we zoom in we can still see a visible difference in error between $\alpha=0.6$ and $\alpha = 1$, it's just the scale is dominated by the error at $\alpha=0.2$. 
%\ambuj{I suggest we omit this last para. I moved one sentence from here to the previous para. That should be enough.}

\subsection{\edit{Effect of Dependence}}
\edit{Next we set up a simulation to study the effect of dependence on lasso estimation error. At time $0$, we set
 $X_0 \sim \mathcal{N}(0, I_{p \times p}),\, Y_0 = \mt{A} X_0 + c\epsilon_0$. For $t \ge 1$, with probability  $\rho$, $ (X_t, Y_t) $ is just a copy of the previous observations, i.e., $Y_t = Y_{t-1},\, X_t = X_{t-1}$. With probability $1-\rho$, $ (X_t, \epsilon_t) $ are fresh independent samples, $X_t  \sim \mathcal{N}(0, I_{p \times p}), \,Y_t  = \mt{A} X_t +  c\epsilon_t$. The settings of $\mt{A}$ and $\epsilon_t$ are exactly the same as above in Subsection~\ref{experiment:longtailed}. We fix shape parameter $\alpha = 1$. Now, let $s = \sqrt{p}$, $p \in \{50,100,150,200\}$, $\rho \in \{0.2,0.4,0.6,0.8\}$, and $T = m \times s\log(p)$ where $m \in \{1, 3, 5, 7, 9, 11, 13, 15, 17, 19 \}$. The average estimation error over 10 repetitions plotted against $\sqrt{\frac{s \log(p)}{T}}$ is shown in Figure~\ref{fig:mixing}. For all choices of the dimension $p$, the results confirm the intuition that higher $\rho$ leads to higher estimation error as higher $\rho$ implies more dependence in data.}

\begin{center}
\begin{figure}[h]\label{fig:mixing}
\includegraphics[width=0.8\textwidth]{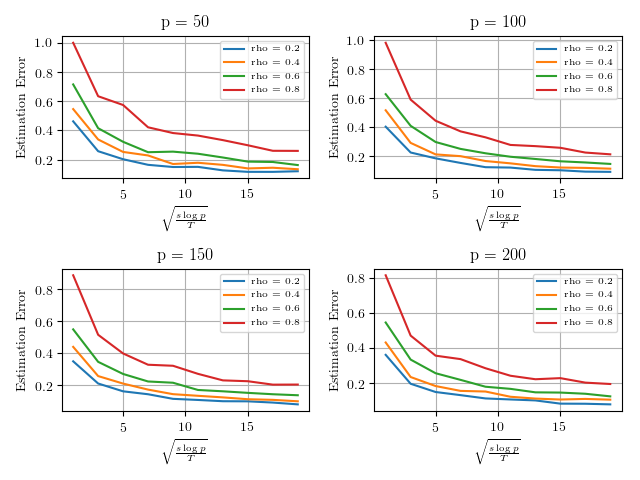}
\caption{Effect of Dependence on Lasso Estimation Error. A larger $\rho$ parameter corresponds to more dependence in the generating process.}
\end{figure}
\end{center}

%% file: appendix.tex
\appendix

\section{Proof of Master Theorem}\label{sec:masterproof}
\input{masterThmProof}
%

%%%%%%%%%%%%%%%%%%%%%%%%%%%%%%%%%%%
\section{Proofs for Gaussian Processes under $\alpha$-Mixing}\label{proof:gauss}
%%%%%%%%%%%%%%%%%%%%%%%%%%%%%%%%%%%
\input{gausProof}

\section{Proofs for Subweibull Random Vectors under $\beta$-Mixing}\label{proof:subweill}
\input{appendix2}

%%
%\section{More Examples}\label{sec:moreexamples}
%
%\input{more_examples}

\section{Verification of Assumptions for the Examples}

\subsection{VAR}\label{veri:VAR}

\input{VAR}
\subsection{VAR with Misspecification}\label{veri:misVAR}

\input{misspecifiedVAR}

\subsection{ARCH}\label{veri:ARCH}

\input{MutltivariateARARCH}

%% file: masterThmProof.tex
%!TEX root = aos-mixing.tex

\begin{proof}[Proof of Theorem \ref{result:master}] 
%\comment{ Added more explanations and corrected a few notataions per reviewer 1's request in pt 8 (Other Comments).}\\
%\ambuj{Ok, great!}

\edit{
The proof follows from optimality of $\bhat$ and the definitions of the RE and DB conditions. 
}

\begin{enumerate}
\edit{
\item 
Since $\bhat$ is optimal for optimization problem~\eqref{eqn:bhat} and $\bstar$ is feasible,
\begin{equation} \label{main:1}
\frac{1}{T} \vertiii{ \mt{Y}-\mt{X}\bhat}_F^2 + \lambda_T \vertii{\vect(\bhat)}_1    \le  
\frac{1}{T} \vertiii{ \mt{Y}-\mt{X}\bstar}_F^2 + \lambda_T \vertii{\vect(\bstar)}_1
\end{equation} 
\item 
Let $\hat{\Delta}:=\bhat-\bstar \in \R^{p\times q}  $. Then, rearranging the inequality~\eqref{main:1} above yields 
\begin{equation} \label{main:2} 
\frac{1}{T}\vertiii{\mt{X}\hat{\Delta}}_F^2 \le \frac{2}{T} \mathrm{tr}(\hat{\Delta}' \mt{X}'\mt{W} ) +\lambda_T\left( \vertii{\vect(\bstar)}_1 - \vertii{\vect(\bhat)}_1\right)
\end{equation}
\item 
%\edit{
Let $S$ denote the support of $\vc{\bstar}$. For any matrix $ \mt{M} $, let $\mt{M}_S$ be the components of $\mt{M} $ restricted to the support $ S $ and similarly for $\mt{M}_{S^C}$; therefore $ \vc{\bstar}= \vc{\bstar}_S + \vc{\bstar}_{S^C} $. With these, note that
%}
\begin{align*}
\vertii{\vect(\bstar+\hat{\Delta})}_1 -\vertii{\vect(\bstar)}_1 \ge &
\{\vertii{\vect(\bstar_S)}_1 - \vertii{\vect(\hat{\Delta}_S)}_1 \} \\
&+ \vertii{\vect(\hat{\Delta}_{S^c})}_1- \vertii{\vect(\bstar)}_1 \\
&=\vertii{\vect(\hat{\Delta}_{S^c})}_1 - \vertii{\vect(\hat{\Delta}_{S})}_1
\end{align*}
\item \label{main:3}
Recall the $RE$ condition with parameters $\alpha$ and $\tau$, and deviation bound condition with constant $\mathbb{Q}(\Sigma_X, \Sigma_W)$. For tuning parameter $\lambda_T\ge 2\mathbb{Q}(\Sigma_X, \Sigma_W)\sqrt{\frac{\log(q)}{T}} $, we have 
\begin{align*}
\alpha \vertiii{\hat{\Delta}}_F^2  -&\tau\Vert \vect(\hat{\Delta} )\Vert_1^2 
\\ &\overset{RE}{\le} 
\frac{1}{T}\vertiii{\mt{X}\Delta }_F^2 
\\&\overset{\eqref{main:2}}{\le }
\frac{2}{T}\mathrm{tr}(\hat{\Delta}' \mt{X}'\mt{W})+  \lambda_T\{\vertii{\vect(\hat{\Delta}_{S})}_1 - \vertii{\vect(\hat{\Delta}_{S^c})}_1\}
 \\&\le 
\frac{2}{T}\sum_{k=1}^{q}\Vert\hat{\Delta}_{:k} \Vert_1 \Vert( \mt{X}'\mt{W})_{:k} \Vert_{\infty} + \lambda_T\{\vertii{\vect(\hat{\Delta}_{S})}_1 - \vertii{\vect(\hat{\Delta}_{S^c})}_1\}
 \\&\le
\frac{2}{T}\Vert\vect(\hat{\Delta}) \Vert_1 \vertiii{\mt{X}'\mt{W} }_{\infty} + \lambda_T\{\vertii{\vect(\hat{\Delta}_{S})}_1 - \vertii{\vect(\hat{\Delta}_{S^c})}_1\}
\\&\overset{DB}{\le}
2\Vert\vect(\hat{\Delta}) \Vert_1\mathbb{Q}(\Sigma_X, \Sigma_W)\mathbb{R}(p,q,T) + \lambda_T\{\vertii{\vect(\hat{\Delta}_{S})}_1 - \vertii{\vect(\hat{\Delta}_{S^c})}_1\}
\\&\le
\Vert\vect(\hat{\Delta}) \Vert_1\lambda_N / 2+ \lambda_T\{\vertii{\vect(\hat{\Delta}_{S})}_1 - \vertii{\vect(\hat{\Delta}_{S^c})}_1\}
\\&\le
\frac{3\lambda_T}{2}\vertii{\vect(\hat{\Delta}_{S})}_1 - \frac{\lambda_T}{2}\vertii{\vect(\hat{\Delta}_{S^c})}_1
\\&\le
2\lambda_T \vertii{\vect(\hat{\Delta})}_1
\end{align*}
\item  \label{main:4}
In particular, this says that 
$
3 \vertii{\vect(\hat{\Delta}_{S})}_1 \ge \vertii{\vect(\hat{\Delta}_{S^c})}_1
$\\
So, we have a lower bound 
 $$\vertii{\vect(\hat{\Delta})}_1\le 4 \vertii{\vect(\hat{\Delta}_S)}_1 \le 4\sqrt{s}\vertii{\vect(\hat{\Delta})}$$
\item \label{upperbound}
Finally, with $\alpha \ge 32s\tau$,
\begin{align*}
\frac{\alpha	}{2}\vertii{\vect(\hat{\Delta} )}_2^2 &\le
( \alpha-16s\tau)\vertii{\vect(\hat{\Delta} )}_2^2  \\
&\le  \alpha \vertii{\vect(\hat{\Delta} )}_2^2  -\tau\Vert \vect(\hat{\Delta} )\Vert_1^2 \\
 &\le 2\lambda_T\Vert \vect(\hat{\Delta} )\Vert_1 \\
 &\le 2\sqrt{s}\lambda_T\Vert \hat{\Delta} \Vert_2
\end{align*}
Thus, we have the upper bound
$$
\vertii{\vect(\hat{\Delta} )}_2 \le \frac{4\lambda_T \sqrt{s}}{\alpha}
$$
\item 
From steps~\eqref{main:3} and~\eqref{main:4}, we have 
\begin{align*}
\frac{1}{T}\vertiii{\mt{X}\hat{\Delta}}_F^2 
\le
8\lambda_T \sqrt{s} \vertii{\vect(\hat{\Delta})}_2
\end{align*}
Together with step \eqref{upperbound}, we obtain that
\begin{align*}
\frac{1}{T}\vertiii{\mt{X}\hat{\Delta}}_F^2 
\le
8\lambda_T \sqrt{s} \vertii{\vect(\hat{\Delta})}_2
\le
32\lambda_T^2 s /\alpha
\end{align*}
}
\end{enumerate}
\end{proof}

%% file: gausProof.tex
%!TEX root = aos-mixing.tex

%\red{TODO: check is these facts are still needed. Probably not...}

%Later we will need the following facts about subgaussian and subexponential random variables/vectors.
%\begin{fact}
%The sub-Gaussian (sub-exponential) ``norm" $ \snorm{\cdot} $($ \enorm{\cdot} $) is indeed a norm on the space of random variables/vectors. 
%\end{fact}
%\begin{fact}\label{fact:subgaussExp}
%A random variable $ U $ is subgaussian, iff $ U^2 $ is subexponential: $\vertii{U^2}_{\psi_1} = \vertii{U}^2_{\psi_2}$.
%\end{fact}

We will also need the following result to control operator norms of matrices in terms of $\ell_1$ norms of the rows and columns.
%\red{TODO: where is this result used in the sequel? I don't think we refer to it even though we use it.}

\begin{fact}[Schur Test]\label{fact:schurTest}
For any matrix $\mt{M}$, we have
\[
\vertiii{ \mt{M} }^2 \le \max_{ i } \vertii{\mt{M}_{i:}}_1 \cdot \max_{ j } \vertii{\mt{M}_{:j}}_1 .
\]
Therefore, for any \emph{symmetric} matrix $ \mt{M} \in \R^{n \times n} $, $ \vertiii{\mt{M}}\le \max_{1\le i\le n} \vertii{\mt{M}_{i:}}_1$.
\end{fact}

\begin{claim}\label{claim:crossCovBdd}
For any random vectors $ X\in \R^n  $ and $ Y\in \R^n  $, we have 
\begin{align*}
\vertiii{\E\bbra{XY'}} = \vertiii{\E\bbra{YX'}}  \le \frac{\vertiii{\Sigma_{X}}+\vertiii{\Sigma_{Y}}}{2}
\end{align*}

\end{claim}

\begin{proof}%[of Claim \eqref{claim:crossCovBdd}]
We have,
\begin{align*}
\vertiii{\E\bbra{XY'}} &=\vertiii{\E\bbra{YX'}}  & &\\
						&:= \sup_{\vertii{u}\le 1,\, \vertii{v}\le 1}\E \bbra{ u'YX'v }\\
						&=  \sup_{\vertii{u}\le 1,\, \vertii{v}\le 1}\E \bbra{ (Y'u)(X'v) }\\
						&\le  \sup_{\vertii{u}\le 1,\, \vertii{v}\le 1} \sqrt{\E\bbra{(Y'u)^2}}\sqrt{\E\bbra{(X'v)^2}}
										&\text{by Cauchy–Schwarz ineq.} \\
						&= \sup_{\vertii{u}\le 1} \sqrt{\E\bbra{(Y'u)^2}}  \sup_{\vertii{v}\le 1} \sqrt{\E\bbra{(X'v)^2}} \\
						&=\sqrt{\vertiii{\E\bbra{XX'}}}\sqrt{\vertiii{\E\bbra{YY'}}}\\
						&\le \frac{\vertiii{\E\bbra{XX'}}+\vertiii{\E\bbra{YY'}}}{2} . \hfill \qedhere
\end{align*}
\end{proof}

The proof of Proposition \ref{result:rhoRE} relies on the following result.
\begin{lm}\label{result:lemma}
For a second order stationary $\rho$-mixing sequence of random vectors $ \{X_t\}_t $, their $l$-th auto-covariance matrix can be bounded as follows:
$$
\vertiii{\Sigma_{{X}}(l)}  \le \rho(l) \vertiii{\Sigma_{{X}}(0) }, \,\forall \, l \in \mathbb{Z}.
$$
\end{lm}
% % % % % % % %
%begin proof
\begin{proof}
%\comment{added a few explanations per reviewer 1's request}\\
%\ambuj{Whatever edits you made in response to reviewer's request, just put them inside {\tt $\backslash$edit}}
\edit{
Recall the definition of $\rho$-mixing. For random vectors $X$ and $Y$ on the probability space $(\Omega,\mathcal{F},\mathcal{P})$, let $\mathcal{A}:= \sigma(X)$ and $\mathcal{B}:= \sigma(Y)$, and $\mathcal{L}^2(\mathcal{C})$ denote the space of square-integrable, $ C $-measurable (real-valued) random variables.
\begin{align*}
	\rho(X,Y)	&:=\sup \{ \mathrm{cor}( f(X),g(Y) )\, | \,\,f \in \mathcal{L}^2(\mathcal{A}), g \in \mathcal{L}^2(\mathcal{B}) \} 
\end{align*}
Then, in particular, we consider the one-dimensional projections and obtain that
\begin{align*}
	\rho(X,Y)					
&\ge \mathrm{cor}(u'X, \, v'Y) 		&&\forall \text{ fixed }\, u,\, v \\
				&=\frac{\vert \mathbb{E}[u'Xv'Y]  \vert}{\sqrt{\mathbb{E}(u'X)^2 \mathbb{E}(v'Y)^2} }\; &&u,\, v\, \text{non-zero} \\
\end{align*}
Re-arranging, $\forall u,\, v$ fixed,
\begin{align*}
\vert u' \mathbb{E}[XY'] v \vert & \le \rho(X,Y) \sqrt{\mathbb{E}(u'X)^2} \sqrt{\mathbb{E}(v'Y)^2} \\
\sup_{u,v} \vert u' \mathbb{E}[XY'] v \vert & \le \rho(X,Y) \sqrt{\sup_u \mathbb{E}(u'X)^2} \sqrt{\sup_v \mathbb{E}(v'Y)^2}
\end{align*}
}
But, 
$$
\vertiii{\mathbb{E}[XY']}  \equiv \sup_{u,v} \vert u' \mathbb{E}[XY'] v \vert  \le \rho(X,Y) \sqrt{\sup_u \mathbb{E}(u'X)^2} \sqrt{\sup_v \mathbb{E}(v'Y)^2}
$$
For a stationary time series $\{X_t \}$, recall that $\forall t,l$
$$
\Sigma_X(l)  := \E[ X_t X_{t+l}'].
$$
By stationarity, $\forall t,l$
$$
\rho(X_t, X_{t+l}) = \rho(l)  .
$$
Hence,
$$
\vertiii{  \Sigma_X(l) }\le \rho(l) \vertiii{\Sigma_X(0) } . \hfill\qedhere
$$
% % % % % % % % 
%end proof 
\end{proof}

%\red{TODO: we need to fix references to Hanson-Wright. In general, we need to make sure there are no more undefined Latex references.}

\begin{proof}[Proof of Proposition \ref{result:rhoDB}]

Note that by Fact \ref{fact:mixingEquiv} $ \alpha $ and $ \rho$-mixing are equivalent for stationary Gaussian processes. The proof will operate via arguments in $ \rho $-mixing coefficients. \\

Recall $\vertiii{\mt{X}'\mt{W} }_{\infty}= \max_{1\le i \le p,1\le j \le q } | [\mt{X}'\mt{W}]_{i,j} |= 
\max_{1\le i \le p,1\le j \le q }\verti{ \mt{X}_{:i}'\mt{W}_{:j}}$. \\

By Assumption \eqref{as:0mean}, we have
\begin{align*}
&\mathbb{E} {\mt{X}_{:i}}=\vc{0},\forall i \;\;\;\text{and} \\
& 
 \mathbb{E} {\mt{Y}_{:j}}=\vc{0},\forall j  
\end{align*}

By first order optimality of the optimization problem in \eqref{eqn:bstar}, we have
$$
\mathbb{E} \mt{X}'_{:i}(\mt{Y}-\mt{X}\bstar)=\vc{0},\forall i \Rightarrow 
\mathbb{E} {\mt{X}_{:i}}'\mt{W}_{:j}=0,\forall i,j 
$$
We know $\forall i,j $
\begin{equation}\label{eq:splitSquare}
\begin{aligned}
\verti{\mt{X}_{:i}'\mt{W}_{:j}} &= \verti{\mt{X}_{:i}'\mt{W}_{:j} - \E[ \mt{X}_{:i}'\mt{W}_{:j} ]}  \\
&=\frac{1}{2}\left| \left(\Vert \mt{X}_{:i}+ \mt{W}_{:j} \Vert^2 - \E[\Vert \mt{X}_{:i}+ \mt{W}_{:j} \Vert^2 ] \right) \right.\\
&\quad \left. - \left( \Vert \mt{X}_{:i} \Vert^2 - \E[\Vert \mt{X}_{:i} \Vert^2 ] \right) - \left( \Vert \mt{W}_{:j} \Vert^2 -\E[ \Vert \mt{W}_{:j} \Vert^2] \right) \right| \\
&\le \half \verti{  \Vert \mt{X}_{:i}+ \mt{W}_{:j} \Vert^2 - \E[\Vert \mt{X}_{:i}+ \mt{W}_{:j} \Vert^2 ]  } \\
&\quad + \half \verti{  \Vert \mt{X}_{:i} \Vert^2 - \E[\Vert \mt{X}_{:i} \Vert^2 ] } + \half \verti{  \Vert \mt{W}_{:j} \Vert^2 -\E[ \Vert \mt{W}_{:j} \Vert^2]  } .
\end{aligned}
\end{equation}
Therefore,
\begin{align*}
&\quad \prob \bpar{ \frac{1}{T} \verti{\mt{X}_{:i}'\mt{W}_{:j}} > 3t } \\
&\le \prob \bpar{ \frac{1}{2T} \verti{  \Vert \mt{X}_{:i}+ \mt{W}_{:j} \Vert^2 - \E[\Vert \mt{X}_{:i}+ \mt{W}_{:j} \Vert^2 ]  } > t }
+ \prob \bpar{ \frac{1}{2T} \verti{  \Vert \mt{X}_{:i} \Vert^2 - \E[\Vert \mt{X}_{:i} \Vert^2 ] } > t } \\
&\quad + \prob \bpar{ \frac{1}{2T} \verti{  \Vert \mt{W}_{:j} \Vert^2 -\E[ \Vert \mt{W}_{:j} \Vert^2] } > t } .
\end{align*}
This suggests a proof strategy where we control each of the three tail probabilities above. Assuming the conditions in Proposition~\ref{result:rhoDB}, we  can apply the Hanson-Wright inequality (Lemma \ref{thm:hanson}) on each of them because we know that

$ \forall i,j $
$$
\mt{X}_{:i} \sim N(0, \Sigma_{\mt{X}_{:i}}),\; i=\wunai p,\; \text{and}
$$ 

$$
\mt{W}_{:j}:= \mt{Y}_{:j}-[\mt{X}\bstar]_{:j} \sim  N(\vc{0}, \Sigma_{\mt{W}_{:j}}),\; j=\wunai q
$$
since both $\{\vc{X}_t\}_{t=1}^T$ and $\{\vc{Y}_t\}_{t=1}^T$ are centered Gaussian vectors. 

So, 
$$
\mt{X}_{:i}+  \mt{W}_{:j}\sim N(0, \Sigma_{\mt{X}_{:i}} + \Sigma_{\mt{W}_{:j}} + \Sigma_{\mt{W}_{:j},\, \mt{X}_{:i}} + \Sigma_{\mt{X}_{:i},\mt{W}_{:j}})
$$
We are ready to apply the tail bound on each term on the RHS of \eqref{eq:splitSquare}.
By Lemma \eqref{thm:hanson}, $\exists $ constant $c>0$ such that $\forall t\ge 0$
\begin{align*}
\frac{\Vert \mt{X}_{:i}\Vert^2  - \mathbb{E}\Vert \mt{X}_{:i} \Vert^2}{T\vertiii{\Sigma_{ \mt{X}_{:i} }}} &\le t   &\text{w.p. at least }1-2\exp(-cT\min\{t,t^2\}) \\
\frac{\Vert  \mt{W}_{:j}\Vert^2  - \mathbb{E}\Vert \mt{W}_{:j} \Vert^2}{T\vertiii{\Sigma_{\mt{W}_{:j} }}} &\le t   &\text{w.p. at least }1-2\exp(-cT\min\{t,t^2\}) \\
\frac{\Vert \mt{X}_{:i}+ \mt{W}_{:j}\Vert^2  - \mathbb{E}\Vert \mt{X}_{:i} + \mt{W}_{:j}\Vert^2}{T\vertiii{\Sigma_{\mt{X}_{:i}+\mt{W}_{:j}} }} &\le t   &\text{w.p. at least }1-2\exp(-cT\min\{t,t^2\}) \end{align*}

With Claim \ref{claim:crossCovBdd}, the third inequality implies, for some $\tilde{c}>0$, that
  w.p. at least $$
  1-8\exp(-\tilde{c}T\min\{t,t^2\})
  $$
  the following holds
$$
\frac{\Vert \mt{X}_{:i}+ \mt{W}_{:j}\Vert^2  - \mathbb{E}\Vert \mt{X}_{:i} + \mt{W}_{:j} \Vert^2}
{2T (\vertiii{ \Sigma_{ \mt{X}_{:i}}  } + \vertiii{\Sigma_{\mt{W}_{:j}} })}\le \frac{\Vert \mt{X}_{:i}+ \mt{W}_{:j}\Vert^2  - \mathbb{E}\Vert \mt{X}_{:i} + \mt{W}_{:j}\Vert^2}{T\vertiii{\Sigma_{\mt{X}_{:i}+\mt{W}_{:j}} }} \le   t   
$$
Therefore,
$$
\frac{\mt{X}_{:i}'\mt{W}_{:j}}{3T (\vertiii{ \Sigma_{ \mt{X}_{:i}}  } +\vertiii{   \Sigma_{\mt{W}_{:j}} })}
\le  3 t   \text{        w.p. at least }1-8\exp(-\tilde{c}T\min\{t,t^2\}) 
$$
Appealing to the union bound over all $i \in [1 \cdots p]$ and $j \in [1 \cdots q]$, for any $ \Delta $
$$
\mathbb{P}[\max_{1\le i\le p, 1\le j\le q} \mt{X}_{:i}'\mt{W}_{:j} \ge \Delta] \le p q\mathbb{P}[	\mt{X}_{:i}'\mt{W}_{:j}\ge \Delta]\\
$$

We can conclude that
$$
\mathbb{P}[\max_{1\le i\le p, 1\le j\le q}\frac{  \mt{X}_{:i}'\mt{W}_{:j}}{3T (\vertiii{ \Sigma_{ \mt{X}_{:i}}  } +\vertiii{   \Sigma_{\mt{W}_{:j}} })} \le 3t]
 \ge 1-8pq\exp(-\tilde{c}T\min\{t,t^2\})
$$
Now, for a free parameter \black{$b>0$}, choose \black{ $t=\sqrt{\frac{(b+1)\log(pq)}{\tilde{c}T}}$}, for $\black{T\ge \frac{(b+1)\log({pq})}{\tilde{c}}}$ we have
\begin{align*}
\mathbb{P}&\left[ \vertiii{\frac{\mt{X}'\mt{W}}{T} }_{\infty} 
\le 
 \sqrt{\frac{(b+1)\log(pq)}{\tilde{c}T}}
\max_{1\le i\le p, 1\le j\le q}(\vertiii{ \Sigma_{ \mt{X}_{:i}}  } +\vertiii{   \Sigma_{\mt{W}_{:j}} })  \right] \\
&\qquad  \ge 
1-8 
\exp[-b\log(pq)]
\end{align*}

Let us find out what $\vertiii{   \Sigma_{\mt{W}_{:j}} }$ and $\vertiii{   \Sigma_{\mt{X}_{:i}} }$ are. Recall $ \mt{W}=\mt{Y}-\mt{X}\bstar $. So, 
\begin{eqnarray} %\label{eq:crosscovSplit}
\Sigma_{\mt{W}_{:i}} = \Sigma_{\mt{Y}_{:i}} + \Sigma_{\mt{X} \bstar_{:i}} + \Sigma_{\mt{Y}_{:i},\,\mt{X} \bstar_{:i}}
+\Sigma_{\mt{X} \bstar_{:i},\,\mt{Y}_{:i}}
\end{eqnarray}
By Claim \ref{claim:crossCovBdd}, we have
\begin{equation}\label{ineq:sigmaW}
\vertiii{\Sigma_{\mt{W}_{:i}}} \le 2\vertiii{\Sigma_{\mt{Y}_{:i}}} + 2\vertiii{\Sigma_{\mt{X} \bstar_{:i}}}  
\end{equation}
Let us figure out each of the summands on the RHS of equation \eqref{ineq:sigmaW} above. 
\begin{align*}
 \Sigma_{\mt{X} \bstar_{:i}} [l,k]  &:= \E \bbra{(\mt{X}\bstar_{:i} )(\mt{X}\bstar_{:i})'}[l,k]\\
&= \E \bbra{e_l'(\mt{X}\bstar_{:i} )(\mt{X}\bstar_{:i})'e_k} \\
&=\E  \bbra{(\mt{X}_{l,:}'\bstar_{:i} )((\bstar_{:i})' \mt{X}_{k,:} )} \\
&= \E  \bbra{(\bstar_{:i} )'\mt{X}_{l,:}  \mt{X}_{k,:}' \bstar_{:i}} \\
&= (\bstar_{:i}) '   \bbra{\E \mt{X}_{l,:}  \mt{X}_{k,:}' }\bstar_{:i}
\end{align*}
With the equality above,
\begin{align*}
\vertiii{\Sigma_{\mt{X} \bstar_{:i}} } &\le \max_{1\le k \le T}\vertii{(\Sigma_{\mt{X} \bstar_{:i}})[{k,:}]}_1 & &\text{by Fact \ref{fact:schurTest}}\\
&\le 2\sum_{l=0}^{T}\rho(l)\vertii{\bstar_{:i}}_2^2 \vertiii{\Sigma_{{X}}(0)} &&\text{by Lemma \ref{result:lemma}}
\end{align*}
Therefore,
$$
\max_{1\le i \le p}\vertiii{\Sigma_{\mt{X} \bstar_{:i}} } \le  \max_{1\le i \le p}2\sum_{l=0}^{T}\rho(l)\vertii{\bstar_{:i}}_2^2 \vertiii{\Sigma_{{X}}(0)}
=
2\vertiii{\Sigma_{{X}}(0)}\sum_{l=0}^{T}\rho(l)\max_{1\le i \le p}\vertii{\bstar_{:i}}_2^2 .
$$
Similarly,
$$
\max_{1\le i \le p}\vertiii{\Sigma_{\mt{Y}_{:i}} } 
\le
2\vertiii{\Sigma_{{Y}}(0)}\sum_{l=0}^{T}\rho(l)
$$
and
$$
\max_{1\le i \le p}\vertiii{\Sigma_{\mt{X}_{:i}} } 
\le
2\vertiii{\Sigma_{{X}}(0)}\sum_{l=0}^{T}\rho(l) .
$$
So, by inequality \eqref{ineq:sigmaW} 
\begin{align*}
\max_{1\le i \le q} 
\vertiii{ \Sigma_{\vc{W}_{:i}} }
 \le
4\sum_{l=0}^{T}\rho(l)\bpar{\vertiii{\Sigma_{{X}}}\max_{1\le i \le p} \vertii{\bstar_{:i}}_2^2 + \vertiii{\Sigma_{{Y}} } } .
\end{align*}

Therefore, 
\begin{align*}
\max_{1\le i\le p, 1\le j\le q}(\vertiii{ \Sigma_{ \mt{X}_{:i}}  } +\vertiii{   \Sigma_{\mt{W}_{:j}} }) 
\le
4\sum_{l=0}^{T}\rho(l) \bpar{
\vertiii{\Sigma_{{X}}}
\bpar{1+ \max_{1\le i \le p}\vertii{\bstar_{:i}}_2^2 }
+ \vertiii{\Sigma_{{Y}} } }
\end{align*}

Finally, we state the final result.
For a free parameter \black{$b>0$}, choose \black{ $t=\sqrt{\frac{(b+1)\log(pq)}{\tilde{c}T}}$}, for $\black{T\ge \frac{(b+1)\log({pq})}{\tilde{c}}}$ we have with probability at least
\[
1-8 
\exp[-b\log(pq)]
\]
that
$$
\vertiii{\frac{\mt{X}'\mt{W}}{T} }_{\infty} 
\le 
 \sqrt{\frac{(b+1)\log(pq)}{\tilde{c}T}}
4\sum_{l=0}^{T}\rho(l) \bpar{
\vertiii{\Sigma_{{X}}}
\bpar{1+ \max_{1\le i \le p}\vertii{\bstar_{:i}}_2^2 }
+ \vertiii{\Sigma_{{Y}} } }
$$
Also, because of Fact \ref{fact:alpharhoEquiv}, we have
$$
\vertiii{\frac{\mt{X}'\mt{W}}{T} }_{\infty} 
\le 
 \sqrt{\frac{(b+1)\log(pq)}{\tilde{c}T}}
8\pi S_{\alpha}(T)
\bpar{
\vertiii{\Sigma_{{X}}}
\bpar{1+ \max_{1\le i \le p}\vertii{\bstar_{:i}}_2^2 }
+ \vertiii{\Sigma_{{Y}} } }
$$
\end{proof}

\begin{proof}[Proof of Proposition \ref{result:rhoRE}]

Note that, by Fact~\ref{fact:alpharhoEquiv}, $ \alpha$ and $ \rho$-mixing are equivalent for stationary Gaussian processes. The proof will operate via arguments involving $ \rho$-mixing coefficients. 

%Denote the design matrix by $\mt{X} = [\tilde{X_1}',\cdots, \tilde{X_N}']',$ $N=T-d+1.$ 
For a fixed unit test vector $\vc{v}\in \R^p $, $\vertii{\vc{v}}_2=1$, consider the Gaussian vector $\mt{X}\vc{v} \in \mathbb{R}^T.$ To apply the Hanson-Wright inequality (Lemma \eqref{thm:hanson}), we have to upper bound the operator norm of the covariance matrix $\mt{Q}$ of $\mt{X}\vc{v}$.

 $\mt{Q}$ takes the form
$$
\mt{Q} = \begin{bmatrix}
\vc{v}'\mathbb{E}X_1X_1'\vc{v}	 & \cdots 	&  	\vc{v}'\mathbb{E}X_1X_j'\vc{v}	& \cdots 	& \vc{v}'\mathbb{E}X_1X_T'\vc{v} \\ 
\vdots							 & 	\ddots		 &  	&  			& \vdots \\ 
 \vc{v}'\mathbb{E}X_tX_1'\vc{v}									&  & \vc{v}'\mathbb{E}X_tX_t' \vc{v} & 		 & \vc{v}'\mathbb{E}X_tX_T' \vc{v} \\ 
\vdots & 				 & & 		 \ddots						 &  	\vdots\\ 
\vc{v}'\mathbb{E}X_T X_1'\vc{v} &\cdots  & \vc{v}'\mathbb{E}X_T X_1'\vc{v} &  \cdots& \vc{v}'\mathbb{E}X_T X_T'\vc{v}\\
\end{bmatrix} 
$$

%\red{TODO: refer to Schur test}
We can thus use Fact~\ref{fact:schurTest} and Lemma~\ref{result:lemma} to upper bound $ \vertiii{\mt{Q}}  $ by
 $$\sum_{t=0}^T\rho(l) \vertiii{\Sigma_X(0) } . $$
%We know $ \vertiii{\Sigma_X(0)} $ is finite since we assume that $X_0$ is finite $a.s.$ element-wise.\\
Now, we can apply Lemma~\ref{thm:hanson} on any fixed unit test vector $ \vc{v} \in \mathbb{R}^{p},\, \vertii{\vc{v}}_2=1$. \\
Recall $\hat{\Gamma} := \frac{\mt{X}'\mt{X}}{T} \in \R^{p\times p}$. Using Lemma~\ref{thm:hanson},  we have, $\forall \eta>0$
\begin{align*}
\mathbb{P} \lbrack |\vc{v}'(\hat{\Gamma} - \Sigma_{X}(0))\vc{v} >\eta \vertiii{\mt{Q}} |  \rbrack &\le 2\exp\{-cT \min(\eta,\eta^2) \} \Rightarrow \\
\mathbb{P} \lbrack \vc{v}'(\hat{\Gamma} - \Sigma_{X}(0))\vc{v} >\eta \sum_{t=0}^T\rho(l) \vertiii{\Sigma_{X}(0)}    \rbrack &\le 2\exp\{-cT \min(\eta,\eta^2) \} .
\end{align*}

Using Lemma F.2 in \cite{basu2015regularized}, for any integer $k>0$, we extend it to all vectors in $\mathbb{J}(2k):=\{ \vc{v}\in \mathbb{R}^{p}: \Vert \vc{v} \Vert \le 1,\Vert \vc{v}\Vert_0 \le 2k  \}$:
\begin{align*}
\mathbb{P} &\left[ \sup_{\vc{v}\in \mathbb{J}(2k)}|\vc{v}'(\hat{\Gamma} - \Sigma_{X}(0))\vc{v}| > \eta \sum_{t=0}^T\rho(l)
\vertiii{\Sigma_{X}(0)  }  \right]\\
 &\le 2\exp\{-cT \min\{\eta,\eta^2\} 
  + 2k\min\{\log(p), \log(\frac{21ep}{2k}))\} .
\end{align*}
%\comment{Edited the ``w.p." statements below to make them look more formal}\\
%\ambuj{Ok. Put any significant new edits inside {\tt $\backslash$edit}}

\edit{
By Lemma 12 in \cite{loh2012high}, we further extend the bound to all $\forall \vc{v} \in \mathbb{R}^{p}$,
\begin{align*}
\prob &\bcur{|\vc{v}'(\hat{\Gamma} - \Sigma_{X}(0))\vc{v} |
>27\eta \sum_{t=0}^T\rho(l)\vertiii{ \Sigma_{X}(0) } \left(\Vert \vc{v}\Vert_2^2 +\frac{1}{k}\Vert \vc{v} \Vert_1^2 \right) 
}
\\
&\le 
 2 \exp\{-cT \min(\eta,\eta^2)  + 2k\min(\log(p), \log(\frac{21ep}{2k}))\} \\
&\Updownarrow\\
\prob &\bcur{
|\vc{v}'(\hat{\Gamma} - \Sigma_{X}(0))\vc{v} |
\le 
27\eta \sum_{t=0}^T\rho(l)\vertiii{ \Sigma_{X}(0) } \left(\Vert \vc{v}\Vert_2^2 +\frac{1}{k}\Vert \vc{v} \Vert_1^2 \right)
}
\\
&> 
1 -   2 \exp\{-cT \min(\eta,\eta^2)  + 2k\min(\log(p), \log(\frac{21ep}{2k}))\} \\
& \Downarrow\\
\prob &\bcur{|\vc{v}'(\hat{\Gamma} )\vc{v} |
> -27\eta \sum_{t=0}^T\rho(l)\vertiii{ \Sigma_{X}(0) } \left(\Vert \vc{v}\Vert_2^2 +\frac{1}{k}\Vert \vc{v} \Vert_1^2 \right) + \lmin{ \Sigma_{X}(0) }\Vert \vc{v}\Vert_2^2 } \\
&>  
1 -   2 \exp\{-cT \min(\eta,\eta^2)  + 2k\min(\log(p), \log(\frac{21ep}{2k}))\}
\end{align*}
}

Intuitively, we know the quadratic form of a Hermitian matrix should have its magnitude bounded from below by its minimum eigenvalue. To achieve that, pick \black{
$\eta =
\frac{\lmin{\Sigma_X(0)}}{54\sum_{t=0}^T\rho(l)\lmax{\Sigma_X(0)}}
$ }. So, we have
$$
|\vc{v}'\hat{\Gamma} \vc{v} |
> \frac{1}{2} \lmin{\Sigma_{X}(0)}  \Vert \vc{v}\Vert_2^2 - 
\frac{\lmin{\Sigma_{X}(0)} }{2k}
\Vert \vc{v} \Vert_1^2 
$$

w.p. 
$$\ge 1 -   2 \exp\{-cT \min(1,\eta^2)  + 2k\min(\log(p), \log(\frac{21ep}{2k}))\}
$$ 
because $\min(1,\eta^2)\le \min(\eta,\eta^2)$.

Now, we choose $k$ to make sure the first component in the exponential dominates. For now, assume $p  \ge \frac{21ep}{2k}$. Let \black{ $k = \ceil{c\frac{T}{4 \log(p)} \min\{1,\eta^2\}}$}. Now, choose $T$ such that $k \ge \frac{21e}{2}$  . Let \black{$T\ge \frac{42 e \log(p)}{c \min\{1, \eta^2\}} $}, where $s$ is the sparsity. 

Finally, we have, for $T\ge s\frac{42 e \log(p)}{c \min\{1, \eta^2\}}$, with probability at least 
$$
1-2\exp\{ -T\frac{c}{2} \min\{1,\eta^2 \} \}{}
$$
the following holds 
$$
|\vc{v}'\hat{\Gamma} \vc{v} |
> \frac{1}{2} \lmin{\Sigma_{X}(0)}  \Vert \vc{v}\Vert_2^2 - 
\frac{\lmin{\Sigma_{X}(0)} }{2k}
\Vert \vc{v} \Vert_1^2  \\
$$

Also, let 
$\tilde{\eta} :=
\frac{\lmin{\Sigma_X(0)}}{108\pi  S_{\alpha}(T)\lmax{\Sigma_X(0)}}
$
we can bound $ \eta $ with $ \tilde{\eta} $ by Fact~\ref{fact:alpharhoEquiv}.
\end{proof}

\section{Hanson-Wright Inequality}
%%%%%%%%%%%%%%%%%%%%%%%%%%%%%%%%%%%%%%%
%We state the Bernstein's inequality \citep[Proposition 5.16]{vershynin2010introduction} below for completeness.
%\begin{pr}[Bernstein's Inequality]\label{thm:Bernstein}
%Let $ X_1, \cdots, X_N $ be independent centered sub-exponential random variables, and $ K = \max_i \enorm{X_i} $
%. Then for every $ a =(a_1, \cdots , a_N ) \in \R^N $ and every $ t \ge 0 $, we have
%\[
%\prob \bcur{
%\verti{
%\sum_{i=1}^{N}a_iX_i
%}
%\ge 
%t
%}
%\le
%2\exp\bbra{
%-C_B\min\bpar{
%\frac{t^2}{K^2\vertii{a}^2_2},
%\frac{t}{K\vertii{a}_{\infty}}
%}
%}
%\]
%where $C_B>0 $ is an absolute constant.
%\end{pr}
%
The general statement of the Hanson-Wright inequality can be found in the paper by \cite{rudelson2013hanson} (see their Theorem 1.1).
%\cite[Theorem 1.1]{rudelson2013hanson}
We use a form of the inequality which is derived in the proof of Proposition 2.4 of \cite{basu2015regularized} as an easy consequence of the general result. We state the modified form of the inequality and the proof below for completeness. 

\begin{lm}[Variant of Hanson-Wright Inequality]\label{thm:hanson}
If $ {Y} \sim N(\vc{0}_{n\times 1},\mt{Q}_{n\times n}) $, then there exists universal constant $ c>0 $ such that for any $ \eta>0 $,
\begin{equation} \label{eq:hanson}
\prob \bbra{
\frac{1}{n}\verti{
\vertii{Y}_2^2
- 
\E \vertii{Y}^2_2
}
>
\eta\vertiii{\mt{Q}}
}
\le
2\exp
\bbra{
-cn \min \bcur{
\eta,\eta^2
}
} .
\end{equation}

\end{lm}
%\begin{rem}
%Note that $ \Tr(\mt{Q}) = \Tr(\E YY') =\E \Tr( YY') =\E \Tr( Y'Y) = \E \Tr\vertii{Y}= \E \vertii{Y} $. So, Eq. \eqref{eq:hanson} can be equivalently stated as 
%\begin{equation}\label{eq:hanson2}
%\prob \bbra{
%\frac{1}{n}\verti{
%\vertii{Y}_2^2
%-  
%
%}
%>
%\eta\vertiii{\mt{Q}}
%}
%\le
%2\exp
%\bbra{
%-cn \min \bcur{
%\eta,\eta^2
%}
%}
%\end{equation}
%\end{rem}
\begin{proof}
The lemma easily follows from Theorem 1.1 in \cite{rudelson2013hanson}. Write $ Y = \mt{Q}^{1/2}X $, where $ X \sim \mathcal{N}(0,\mt{I}) $ and $ (\mt{Q}^{1/2})'(\mt{Q}^{1/2}) = \mt{Q} $. Note that each component $ X_i $ of $ X $ is independent $ \mathcal{N}(0, 1) $, so that $ \snorm{X_i} \le 1 $.
Then, by the above theorem,
\begin{align*}
\prob \bbra{
\frac{1}{n}\verti{
\vertii{Y}_2^2
- \Tr(\mt{Q})
}
>
\eta\vertiii{\mt{Q}}
}
&= 
\prob \bbra{
\frac{1}{n}\verti{
X'\mt{Q}X
- 
\E X'\mt{Q}X
}
>
\eta\vertiii{\mt{Q}}
} \\
&\le 
2\exp
\bbra{
-c \min \bcur{
\frac{n^2 \eta^2\vertiii{\mt{Q}}}{\vertiii{\mt{Q}}_F^2},
\frac{n \eta\vertiii{\mt{Q}}}{\vertiii{\mt{Q}}}
} 
} \\
&\le
2\exp
\bbra{
-c \min \bcur{
\eta,\eta^2
}
} \quad \quad \text{since $ \vertiii{\mt{Q}}_F^2 \le n\vertiii{\mt{Q}}^2 $}
\end{align*}
Lastly, note that $ \Tr(\mt{Q}) = \Tr(\E YY') =\E \Tr( YY') =\E \Tr( Y'Y) = \E \Tr\vertii{Y}^2= \E \vertii{Y}^2 $.
\end{proof}

%% file: appendix2.tex
%!TEX root = aos-mixing.tex

%\red{\todo{own citation: 2017 or 17?}}
\subsection{Proof Related to Subweibull Properties}

\begin{proof}(of Lemma \ref{lem: subWei})
\emph{Property 1 $\Rightarrow$ Property 2:}
Since we can scale $X$ by $K_1$, without loss of generality, we can assume $ K_1 =1  $. Then we have, for $p \ge \gamma$,
\begin{align*}
\E \verti{X}^p 	&= \int_{0}^{\infty}\mathbb{P} \bpar{ \verti{X}^p \ge u } du \\
				&= \int_{0}^{\infty}\mathbb{P} \bpar{ \verti{X} \ge t } pt^{p-1 }dt  && \text{using change of variable } u=t^p\\
				&\le \int_{0}^{\infty} 2 e^{-t^{\gamma}} pt^{p-1} dt   && \text{by Property 1} \\
				&= \frac{2p}{\gamma} \int_{0}^{\infty} e^{-v} \cdot v^{\frac{p-1}{\gamma}} v^{\frac{1-\gamma}{\gamma}} dv &&\text{using change of variable } v=t^{\gamma}\\
				&= \frac{2p}{\gamma} \int_{0}^{\infty} e^{-v} \cdot v^{p/ \gamma - 1} dv\\		
				&= \frac{2p}{\gamma} \cdot\Gamma\left(  \frac{p}{\gamma}\right)\\				
			&\le \frac{2p}{\gamma} \left(  \frac{p}{\gamma}\right)^{p/\gamma} &&\text{since } \Gamma(x) \le x^x,\, \forall x\ge 1 .
\end{align*} 
Therefore, for $p \ge \gamma$,
\begin{align*}
\bpar{\E \verti{X}^p}^{1/p} \le 2^{1/p} {(1/\gamma)}^{1/p} p^{1/p} \bpar{p/\gamma}^{1/\gamma}\le C_\gamma \cdot p^{1/\gamma}
\end{align*}
where $C_\gamma = 4(1/\gamma \vee 1) (1/\gamma)^{1/\gamma}$. If $\gamma \le 1$, this covers all $p \ge 1$. If $\gamma > 1$,
we have, for $p = 1,\ldots,\lceil \gamma \rceil - 1$,
\begin{align*}
\bpar{\E \verti{X}^p}^{1/p} \le 2^{1/p} {(1/\gamma)}^{1/p} p^{1/p} \max_{i=1,\ldots,\lceil \gamma \rceil - 1}\ \Gamma(i/\gamma)^{1/i} \le C'_\gamma ,
\end{align*}
where $C'_\gamma = 4(1/\gamma \vee 1) \max_{i=1,\ldots,\lceil \gamma \rceil - 1}\ \Gamma(i/\gamma)^{1/i}$. Therefore, for all $p$,
\begin{align*}
\bpar{\E \verti{X}^p}^{1/p} \le (C_\gamma \vee C'_\gamma) \cdot p^{1/\gamma} .
\end{align*}

\emph{Property 2 $\Rightarrow$ Property 3:}
Without loss of generality, we can assume $ K_2 =1  $. Using Taylor series expansion of $ \exp(\cdot) $, for some positive $ \lambda $,
\begin{align*}
\E \exp\bbra{\bpar{\lambda\verti{X}}^{\gamma_2} }
	&= \E \bbra{ 1 + \sum_{p=1}^{\infty}\frac{\E \bbra{\bpar{\left( \lambda  \verti{X} \right)^{\gamma}}^p}}{p!}} \\
	&\le 1 + \sum_{p=1}^{\infty} \frac{\bpar{\lambda^{\gamma }\gamma p}^p}{\bpar{p/e}^p} && \text{by Property 2 and Stirling's approx.  } \\
	&=\sum_{p=0}^{\infty}\bpar{e\gamma \lambda^{\gamma}}^p = \frac{1}{1- e \gamma \lambda^{\gamma}} \le 2 ,
\end{align*}
where the last inequality holds for any $ \lambda  $ satisfying $ e \gamma \lambda^{\gamma} \le 1/2$, i.e., ${\lambda} \le \bpar{2e\gamma}^{-1/{\gamma}} $. Therefore Property 2 holds with $ K_3 = \bpar{2e\gamma}^{1/{\gamma}} $.

\emph{Property 3 $\Rightarrow$ Property 1:}
Without loss of generality, we can assume $ K_3 =1  $. For all $ t>0 $,
\begin{align*}
\prob\left(  \verti{X}>t \right)&=\prob \left(  \exp\left(  \verti{X}^{\gamma_2}\right) \ge \exp\left(   t^{\gamma_2}\right) \right) \\
&\le  \exp \left(  -\left(   t^{\gamma_2}\right)\right)\E \exp\left(  \verti{X}^{\gamma_2}\right) &&\text{by Markov's inequality}  \\
&\le 2\exp \left(  -\left(   t^{\gamma_2}\right) \right)&&\text{by Property 3.} \qedhere
\end{align*}
\end{proof}

\begin{proof}(of Lemma~\ref{lemma:gammaNormOfSquares})
By definition,
\begin{align*}
\vertii{X^2}_{\psi_\gamma} &= \sup_{p\ge 1} p^{-1/\gamma}\bpar{\E \verti{X^2}^p}^{1/p} \\
&= \sup_{p\ge 1} \bpar{ p^{-1/(2\gamma)}\bpar{\E \verti{X}^{2p}}^{1/2p}}^2 
\end{align*}
Now we make a change of variables $ \tilde{p} := 2p $. Then, we have,
\begin{align*}
\vertii{X^2}_{\psi_\gamma}  &= 2^{1/\gamma} \sup_{\tilde{p}\ge 2} \bpar{ \tilde{p}^{-1/(2\gamma)}\bpar{\E \verti{X}^{\tilde{p}}}^{1/\tilde{p}}}^2 \\
&\le  2^{1/\gamma} \sup_{\tilde{p}\ge 1} \bpar{ \tilde{p}^{-1/(2\gamma)}\bpar{\E \verti{X}^{\tilde{p}}}^{1/\tilde{p}}}^2 \\
&=  2^{1/\gamma}\bpar{  \sup_{\tilde{p}\ge 1} \tilde{p}^{-1/(2\gamma)}\bpar{\E \verti{X}^{\tilde{p}}}^{1/\tilde{p}}}^2 \\
&= 2^{1/\gamma} \vertii{X}_{\psi_{2\gamma}}^2 . \qedhere
\end{align*}
%Let $ \tilde{\gamma}_2 = 2 \gamma_2 $ 
%%and $ \tilde{t }= \sqrt{t} $
%, then
%\begin{align*}
%\vertii{X}_{\tilde{\gamma}_2} &=\inf\bcur{{s}>0 : \E \exp \bpar{\verti{X} /{s}}^{\tilde{\gamma}_2} \le 2} \\
%&=\inf\bcur{{s}>0 : \E \exp \bpar{\verti{X} /{s}}^{2{\gamma}_2} \le 2} \\
%&=\inf\bcur{{s}>0 : \E \exp \bpar{\verti{X}^2 /{s}^2}^{{\gamma}_2} \le 2} \\
%&= \sqrt{\vertii{X^2}_{\gamma_2}}
%\end{align*}
\end{proof}

\subsection{Subweibull Norm Under Linear Transformations}

We will need the following result about changes to the subweibull norm under linear transformations.

\begin{lm}
\label{result:subweibulllinear}
Let $X$ be a random vector and $\mt{A}$ be a fixed matrix. We have,
\[
\norm[\mt{A} X]{\psi_\gamma} \le \vertiii{\mt{A}} \cdot \norm[X]{\psi_\gamma}
\]
\end{lm}
\begin{proof}
We have,
\begin{align*}
\vertii{\mt{A} X}_{\psi_\gamma} &= 
\sup_{\vertii{v}_2\le 1}\vertii{v'\mt{A} X}_{\psi_\gamma} = \sup_{\vertii{v}_2\le 1}\vertii{(\mt{A}' v)' X}_{\psi_\gamma}\\
&\le \sup_{\vertii{u}_2\le \vertiii{\mt{A}}}\vertii{u' X}_{\psi_\gamma}\\
&= \vertiii{\mt{A}} \sup_{\vertii{u}_2\le 1}\vertii{u' X}_{\psi_\gamma} =  \vertiii{\mt{A}}\vertii{X}_{\psi_\gamma} . \qedhere
\end{align*}
\end{proof}

\subsection{Concentration Inequality for Sums of $\beta$-Mixing Subweibull Random Variables}
\label{sec:betamixingsubweibull}

We will state and prove a modified form of Theorem 1 of~\cite{merlevede2011bernstein}. This concentration result will be used to prove the high probability guarantees on the deviation bound (Proposition~\ref{Lemma:SubweibullDevBound}) and lower restricted eigenvalue
 (Proposition~\ref{Lemma:SubweibullRE}) conditions.

\begin{lm}\label{lemma:convenient form}
Let $ (X_j)_{j=1}^T $ be a strictly stationary sequence of zero mean random variables that are subweibull($\gamma_2$) with subweibull constant $K$. Denote their sum by $S_T$. Suppose their $\beta$-mixing coefficients
satisfy $\beta(n) \le 2 \exp(-c n^{\gamma_1})$.
Let $ \gamma$ be a parameter given by
\begin{align*}
\frac{1}{\gamma} = \frac{1}{\gamma_1} + \frac{1}{\gamma_2} .
\end{align*}
Further assume $ \gamma <1 $. Then for $ T>4 $, and any $ t> 1/T $ ,
\begin{align*}
\mathbb{P} \bcur{ \verti{ \frac{S_T}{T} } >t} 
%&\le
%\exp \bcur{   \log(T) -\frac{(tT)^{\gamma}}{C_1}} + \exp\bcur{ -\frac{(tT)^2}{C_2(1+TV)}} \\
%&\quad + \exp \bcur{ -\frac{(tT)^2}{C_3T} \exp \bcur{\frac{1}{C_4}\bpar{\frac{(tT)^{(1-\gamma)}}{(\log (tT))}}^{\gamma}}}\\
&\le
T \exp \bcur{   -\frac{(tT)^{\gamma}}{K^\gamma C_1}} + \exp\bcur{ -\frac{t^2T}{K^2 C_2}} \numberthis \label{eq:convenientConcentration}
\end{align*}
%where \green{$ C_5:= \omega_4 \cdot\max\{ C_2,C_3\} V(Y^2,\gamma_1)  $} for some constant $ \omega_4 $ depending on $ \gamma$.
where the constants $C_1, C_2$ depend only on $\gamma_1, \gamma_2$ and $c$.
\end{lm}
\begin{proof}
Note that, in this proof, constants $C, C_1, C_2, \ldots$ can depend on $c, \gamma_1$ and $\gamma_2$ and $C_1, C_2$ in the proof are not the same as the eventual constants $C_1, C_2$ that appear in the lemma statement.

Further, we will assume that $K=1$. The general form then follows by scaling the random variables by $1/K$ and applying the lemma with $t$ replaced by $t/K$.
The proof consists of two parts. First, we will state a concentration inequality of~\cite{merlevede2011bernstein} and bound a certain parameter $ V $ appearing in their inequality using the $\beta $-mixing assumption. Second, we will simplify the
expression that we get directly from their concentration inequality to get a more convenient form.

\paragraph{\textbf{Step 1: Controlling the $ V $ parameter using $ \beta $-mixing coefficients}}
First, recall that Theorem 1 of~\cite{merlevede2011bernstein}, under the condition of our lemma, gives
\begin{align*}
\mathbb{P} \bcur{ \verti{ S_T} >u} 
&\le
T \exp \bcur{   -\frac{u^{\gamma}}{C_1}} + 
\exp\bcur{ -\frac{u^2}{C_2(1+TV)}} \\
&\quad+ \exp \bcur{ -\frac{u^2}{C_3T} \exp \bcur{\frac{1}{C_4}\bpar{\frac{u^{(1-\gamma)}}{\log (u)}}^{\gamma}}} \numberthis \label{eq:MerlevedeConcentration} 
%&\le
%\exp \bcur{   \log(T) -\frac{(u)^{\gamma}}{C_1}} + \exp\bcur{ -\frac{t^2T}{C_5}} 
\end{align*}

First of all, we need to control the quantity $ V $ that appears in the denominator of the second term of \eqref{eq:MerlevedeConcentration}.{ $ V $ is a worst case measure of the partial sum of the  auto-covariances on the clipped dependent sequence $ (X_t)_{t=1}^T $. It is increasing in time horizon $ T $ and related to dimension $ p $ and sparsity $ s $ and hence not an absolute constant. To the best of our knowledge, $ V $ is not controllable under the weaker $ \alpha$-mixing condition.}  As \cite{merlevede2011bernstein} mention in their Section 2.1.1, using results of \cite{viennet1997inequalities}, we have, for any $ \beta $-mixing strictly stationary sequence $ (Y_t) $ with geometrically decaying $ \beta $-mixing coefficients; i.e.,
\begin{align*}
\beta(k) \le 2\exp\bcur{-c k^{\gamma_1}} \;\text{for any positive }k
\end{align*}
the associated quantity $V$ can be upper bounded as
\begin{align*}
V\le \E X_{1}^2  + 4\sum_{k\ge 0} \E (B_k X_{1}^2)
\end{align*}
for some sequence $ (B_k) $ with values in $ [0,1] $ satisfying $ \E{(B_k)} \le \beta(k) $. In our case, $ (X_t) $ is stationary and we know that its finite moments exist because of Assumption~\ref{assump:subW}. Then,
\begin{align*}
V&\le \E X_{1}^2  + 4\sum_{k\ge 0} \E (B_k X_{1}^2) \\
&\le 
\E X_{1}^2  + 4\sum_{k\ge 0} \sqrt{\E (B_k^2) \E( X_{1}^4) } &&\text{Cauchy-Schwarz ineqeuality} \\
&=
\E X_{1}^2  + 4\sqrt{\E( X_{1}^4)}\sum_{k\ge 0} \sqrt{\E (B_k^2) } &&\text{all finite moments of } X_{1} \text{ exist  }\\
&\le
\E X_{1}^2  + 4\sqrt{\E( X_{1}^4)}\sum_{k\ge 0} \sqrt{\E (B_k) }\\
&\le C ,
\end{align*}
where the second to last inequality follows because $ B_k \in [0,1] \Rightarrow B_k^2 \le B_k $. The last inequality comes from the fact that $ \sqrt{\E{(B_k)}} \le \sqrt{\beta(k) } \le \sqrt{2}\exp\bcur{-\half c k^{\gamma_1}} \Rightarrow (\sqrt{\E{(B_k)}})$ summable. Moreover since $X_1$ is subweibull($\gamma_2$) with constant $1$,
both $\E X_{1}^2$ and $\E( X_{1}^4)$ are bounded with constants depending only on $\gamma_2$. Note that $ C $ depends on $ c, \gamma_1 $ and $\gamma_2$.

\paragraph{\textbf{Step 2: Deriving a Convenient form}}
Eventually we will apply the concentration inequality above with $ u=tT $, and we will choose $ t $ such that $ u=tT  > 1$.
Under the condition that $u > 1$, we will now show that the term appearing in the exponent
in the third term in~\eqref{eq:MerlevedeConcentration},
\begin{align}\label{eq:thrdterm}
\frac{(u)^{(1-\gamma)}}{(\log (u))} ,
\end{align}
is larger than a $\gamma$-dependent constant. Along with the fact that $V$ is a constant in the second term, the second and third terms in~\eqref{eq:MerlevedeConcentration} can then be combined into one. 

Let $ u >1$. Note that the expression~\eqref{eq:thrdterm} remains positive and blows up to infinity as $ u  $ approaches $ 1 $ from above. Taking derivative with respect to $ u $, we obtain
\begin{align*}
\frac{d}{du}\frac{u^{(1-\gamma)}}{(\log (u))} & = \frac{u^{-\gamma}}{\log(u)} \bbra{ (1-\gamma) -\frac{1}{\log(u)}  }
\end{align*}

Observe that the derivative is negative when $ u < u^\ast= e^{\frac{1}{1-\gamma}} $; for $ u >u^\ast $, it becomes positive again. Hence, the expression~\eqref{eq:thrdterm} reaches its minimum at $ u^\ast $, where its value is,
\begin{align*}
\frac{\bpar{e^{\frac{1}{1-\gamma}}}^{1-\gamma}}{\frac{1}{1-\gamma}} = e(1-\gamma) ,
\end{align*}
which is positive since $\gamma < 1$.
\end{proof}

\subsection{Proofs of Deviation and RE Bounds}

\begin{proof}(of Proposition~\ref{Lemma:SubweibullDevBound})
Note that constants $C_1, C_2, \ldots$ can change from line to line and depend only on $\gamma_1, \gamma_2, c$ appearing in Assumption~\ref{assump:betamixing} and
Assumption~\ref{assump:subW}, and on the constant $c'$ appearing in the high probability guarantee.

Recall that $ \mt{W }:=  \mt{Y} -\mt{X}\bstar $, and  $\vertiii{\mt{X}'\mt{W} }_{\infty}= \max_{1\le i \le p,1\le j \le q } | [\mt{X}'\mt{W}]_{i,j} |= 
\max_{1\le i \le p,1\le j \le q }\verti{ (\mt{X}_{:i})'\mt{W}_{:j}}$. \\

By Assumption~\ref{as:0mean}, we have
\begin{align*}
&\mathbb{E} {\mt{X}_{:i}}=\vc{0},\forall i =1,\cdots, p\;\;\;\text{and} \\
& 
 \mathbb{E} {\mt{Y}_{:j}}=\vc{0},\forall j  =1,\cdots, q
\end{align*}

By first order optimality of the optimization problem in \eqref{eqn:bstar}, we have
$$
\mathbb{E} (\mt{X}_{:i})'(\mt{Y}-\mt{X}\bstar)=\vc{0},\forall i \;\\
\Rightarrow 
\mathbb{E} ({\mt{X}_{:i}})'\mt{W}_{:j}=0,\forall i,j 
$$
We know $\forall i,j $
\begin{align*}
&\quad \verti{(\mt{X}_{:i})'\mt{W}_{:j}} \\
&= \verti{(\mt{X}_{:i})'\mt{W}_{:j} - \E[ (\mt{X}_{:i})'\mt{W}_{:j} ]}  \\
&=\frac{1}{2} \left| \left(\Vert \mt{X}_{:i}+ \mt{W}_{:j} \Vert^2 - \E[\Vert \mt{X}_{:i}+ \mt{W}_{:j} \Vert^2 ] \right) \right. \\
&\quad \left. - \left( \Vert \mt{X}_{:i} \Vert^2 - \E[\Vert \mt{X}_{:i} \Vert^2 ] \right) 
- \left( \Vert \mt{W}_{:j} \Vert^2 -\E[ \Vert \mt{W}_{:j} \Vert^2] \right) \right| \\
&\le \half \verti{  \Vert \mt{X}_{:i}+ \mt{W}_{:j} \Vert^2 - \E[\Vert \mt{X}_{:i}+ \mt{W}_{:j} \Vert^2 ]  } \\
&\quad + \half \verti{  \Vert \mt{X}_{:i} \Vert^2 - \E[\Vert \mt{X}_{:i} \Vert^2 ] } + \half \verti{  \Vert \mt{W}_{:j} \Vert^2 -\E[ \Vert \mt{W}_{:j} \Vert^2]  } 
\end{align*}
Therefore,
\begin{align*}
&\quad \prob \bpar{ \frac{1}{T} \verti{(\mt{X}_{:i})'\mt{W}_{:j}} > 3t } \\
&\le \prob \bpar{ \frac{1}{2T} \verti{  \Vert \mt{X}_{:i}+ \mt{W}_{:j} \Vert^2 - \E[\Vert \mt{X}_{:i}+ \mt{W}_{:j} \Vert^2 ]  } > t }
+ \prob \bpar{ \frac{1}{2T} \verti{  \Vert \mt{X}_{:i} \Vert^2 - \E[\Vert \mt{X}_{:i} \Vert^2 ] } > t } \\
&\quad + \prob \bpar{ \frac{1}{2T} \verti{  \Vert \mt{W}_{:j} \Vert^2 -\E[ \Vert \mt{W}_{:j} \Vert^2] } > t } 
\end{align*}

We will now control each of the three the tail probabilities above. Before we apply Lemma~\ref{lemma:convenient form}, we have to figure out their subweibull norms.
We will first calculate the subweibull($\gamma_2$) norm of $\mt{X}_{ti}, \mt{W}_{tj}$ and $\mt{X}_{ti} + \mt{W}_{tj}$. This will immediate yield control of the subweibull($\gamma_2/2$) norms of their
squares via Lemma~\ref{lemma:gammaNormOfSquares}.

Recall that
\begin{align*}
{\mt{W}_{t:}} &=\mt{Y}_{t:} -(\mt{X}\bstar)_{t:} \\
				&=\mt{Y}_{t:} -{\mt{X}_{t:}}\bstar 
\end{align*}
Therefore, we have,
\begin{align*}
\gnorm[\mt{W}_{tj} ]
&\le \knorm{\mt{W}_{t:}}  & &\text{by Definition }\ref{def:vector} \\
&= \knorm{\mt{Y}_{t:} -{\mt{X}_{t:}}\bstar} \\
&\le \knorm{\mt{Y}_{t:}} +\knorm{{\mt{X}_{t:}}\bstar} & &\knorm{\cdot} \text{ is a norm} \\
&\le \knorm{\mt{Y}_{t:}} +\knorm{{\mt{X}_{t:}}}\vertiii{\bstar} && \text{by Lemma }\ref{result:subweibulllinear} \\
&\le K_Y +\vertiii{\bstar}  K_X & &\text{by Assumption }\ref{assump:subW} .
\end{align*}
We also have,
\begin{align*}
\knorm{\mt{X}_{ti} + \mt{W}_{tj}} &\le \knorm{\mt{X}_{ti}} + \knorm{\mt{W}_{tj}} &&\knorm{\cdot} \text{ is a norm} \\
						&\le {K_{Y}} + {K_{X}}\bpar{1+\vertiii{\bstar} } .
\end{align*}
Using Lemma~\ref{lemma:gammaNormOfSquares}, we know that the subweibull($\gamma_2/2$) constants of
the squares of $\mt{X}_{ti}, \mt{W}_{tj}$ and $\mt{X}_{ti} + \mt{W}_{tj}$ are all bounded by
\begin{align*}
 K = 2^{2/\gamma_2} \bpar{{K_{Y}} + {K_{X}}\bpar{1+\vertiii{\bstar} }  }^2 .
\end{align*} 
We now apply Lemma~\ref{lemma:convenient form} three times with $\gamma_2$ replaced by $\gamma_2/2$, to get, for any $t > 1/2T$,
\begin{align*}
\prob \bpar{ \frac{1}{T} \verti{(\mt{X}_{:i})'\mt{W}_{:j}} > 3t } 
&\le 3T \exp \bcur{   -\frac{(2tT)^{\gamma}}{K^\gamma C_1}} + 3\exp\bcur{ -\frac{4t^2T}{K^2 C_2}},
\end{align*}
where $\gamma = (1/\gamma_1 + 2/\gamma_2)^{-1}$ is less than $1$ by Assumption~\ref{assump:gammaSmall}.

Now, taking a union bound over the $pq$ possible values of $i, j$, gives us
\begin{align*}
\prob \bpar{ \frac{1}{T} \vertiii{\mt{X}'\mt{W} }_{\infty} > 3t } 
&\le 3T pq \exp \bcur{   -\frac{(2tT)^{\gamma}}{K^\gamma C_1}} + 3pq\exp\bcur{ -\frac{4t^2T}{K^2 C_2}} .
\end{align*}
If we set,
\[
t = K \max\left\{ C_2 \sqrt{\frac{\log(3pq)}{T}} , \frac{C_1}{T} \left( \log(3 T p q) \right)^{1/\gamma} \right\}
\]
then the probability of the large deviation event above is at most $$2 \exp(- c' \log (3 p q)) .$$ Note that the constant $c'$ can be made arbitrarily large
but affects the constants $C_1, C_2$ above. 

In the expression for $t$ above, we want to ensure that two conditions are met. First, the $1/\sqrt{T}$ term should dominate. 
That is, we want,
\[
\sqrt{ \frac{\log(3pq)}{T}} \ge \frac{C_1}{T} \left( \log(3 T p q) \right)^{1/\gamma} ,
\]
which, in turn, is implied by
\[
\sqrt{ \frac{\log(3pq)}{T}} \ge \frac{C_2}{T} \left( \log(3 T ) \right)^{1/\gamma}
\quad
\text{ and }
\quad
\sqrt{ \frac{\log(3pq)}{T}} \ge \frac{C_2}{T} \left( \log(pq) \right)^{1/\gamma}
\]
Both of these are met if $T \ge C_3 (\log(pq))^{\tfrac{2}{\gamma}-1}$.

Finally, the condition $t > 1/2T$ should be met.That is,
\[
C_2 \sqrt{\frac{\log(3pq)}{T}} > \frac{1}{2T}
\]
which happens as soon as $T \ge C_2^2/4$.
\end{proof}

\begin{proof}(of Proposition~\ref{Lemma:SubweibullRE})
Recall that  $ X_1, \cdots, X_t  \in \R^p$ are subweibull random variables forming a $ \beta $-mixing and stationary sequence.

\paragraph{\textbf{Step I: Concentration for a fixed vector}}

Now, fix a unit vector $ v\in \R^p,\; \vertii{v}_2 =1 $.
Define real valued random variables $ Z_t = v' X_t,\, t=1, \cdots, T $. Note that the $ \beta $-mixing rate of
$ (Z_t) $ is bounded by the same of $ (X_t) $ by Fact~\ref{fact:mixingEquiv}. From Lemma~\ref{lemma:gammaNormOfSquares}, we know that $\vertii{Z_t^2}_{\psi_{\gamma_2/2}}  \le 2^{2/\gamma_2} \vertii{Z_t}^2_{\psi_{\gamma_2}}$. Moreover, $\vertii{Z_t}_{\psi_{\gamma_2}} \le \vertii{X_t}_{\psi_{\gamma_2}}$. Therefore, we can invoke Lemma~\ref{lemma:convenient form} for the sum $ S_T(v) = \sum_{t=1}^{T} \bpar{ Z_t^2 -\E Z^2_t }$ with $\gamma_2 $ replaced by $\gamma_2/2$, $\gamma = \left( 1/\gamma_1 + 2/\gamma_2 \right)^{-1}$ and $K = 2^{2/\gamma_2} K_X^2$ to get the following bound, for $T > 4$ and $t > 1/T$,
\[
\mathbb{P} \bcur{ \verti{ \frac{S_T(v)}{T} } >t} 
\le
T \exp \bcur{   -\frac{(tT)^{\gamma}}{K^\gamma C_1}} + \exp\bcur{ -\frac{t^2T}{K^2 C_2}}
\]

\paragraph{\textbf{Step II: Uniform concentration over all vectors}}
Let $  \mathbb{J}(2k) $ denote the set of $2k$-sparse vector with Euclidean norm at most $1$. Then, using union bound arguments similar to those in Lemma F.2 of \cite{basu2015regularized}, we have
\begin{align*}
\mathbb{P} &\bcur{ \sup_{v \in \mathbb{J}(2k)}\verti{ \frac{S_T(v)}{T} } > 3t }\\
&\quad \le
\exp  \bcur{  \log(T) -\frac{(tT)^{\gamma}}{K^\gamma C_1} +k\log(p)} 
+ \exp\bcur{ -\frac{t^2 T}{K^2 C_2}+k\log(p)} .
\end{align*}
From the $ 2k $-sparse set, we will extend our bound to all $ v \in \R^p $. To do so, we will apply Lemma 12 in~\cite{loh2012high}. For $ k\ge1 $, with probability at least
\begin{align}\label{eq:failure}
 1- \exp  \bcur{  \log(T) -\frac{(tT)^{\gamma}}{K^\gamma C_1} +k\log(p)} 
- \exp\bcur{ -\frac{t^2 T}{K^2 C_2}+k\log(p)}
\end{align}
the following holds uniformly for all $ v \in \R^p $
\begin{align*}
\frac{1}{T} \verti{S_T(v)} \ge 27t \bpar{ \vertii{v}^2_2 + \frac{1}{k}\vertii{v}^2_1 } .
\end{align*}
Let $ \hat{\Sigma}_T(v) := \frac{1}{T}\vertii{\mt{X}v}^2_2$ and note that $ \E \hat{\Sigma}_T(v) = v'\Sigma_X(0)v  $. Therefore, $ \frac{1}{T} S_T = \hat{\Sigma}_T(v) -\E \hat{\Sigma}_T(v)  $. Using these notations, the above inequality implies that
\begin{align*}
\hat{\Sigma}_T(v) &\ge  v'\bpar{\Sigma_X(0)}v -27\cdot t \bpar{\vertii{v}^2_2 + \frac{1}{k}\vertii{v}^2_1 }\\
&\ge \lmin{\Sigma_X(0)} \vertii{v}^2_2 -27\cdot t \bpar{\vertii{v}^2_2 + \frac{1}{k}\vertii{v}^2_1 } \\
&= \vertii{v}^2_2 \bpar{\lmin{\Sigma_X(0)}  -27t } -\frac{27t}{k}\vertii{v}^2_1 \\
&= \vertii{v}^2_2  \half \lmin{\Sigma_X(0)} -\frac{\lmin{\Sigma_X(0)}}{2k}\vertii{v}^2_1 ,
\end{align*}
where the last line follows by picking $t = \frac{1}{54}\lmin{\Sigma_X(0)} $. 

\paragraph{\textbf{Step III: Selecting parameters}} The only thing left is to set the parameter $k$ appropriately. We want to set it so that
\begin{align*}
  2k \log p = \min\left\{ \frac{(tT)^{\gamma}}{K^\gamma C_1} , \frac{t^2 T}{K^2 C_2} \right\}
\end{align*}
so that the failure probability in~\eqref{eq:failure} is at most $1-2T\exp(-k \log p)$.
We want the minimum above to be attained at the first term which means we want
\[
T \ge \left(\frac{K}{t}\right)^{\frac{2-\gamma}{1-\gamma}} \left(\frac{C_2}{C_1}\right)^{\frac{1}{1-\gamma}}
\]
Under this condition, we have
\[
k = \frac{(tT)^{\gamma}}{2 K^\gamma C_1 \log p} .
\]
To ensure that $ k\ge 1 $, we need
\begin{align*}
T&\ge \frac{54 K \bpar{2  C_1\log(p)}^{1/\gamma}}{\lmin{\Sigma_X(0)}}
\end{align*}

To conclude, we have the following RE guarantee. For sample size
\begin{align*}
T \ge \max \bcur{
\frac{54 K \bpar{2  C_1\log(p)}^{1/\gamma}}{\lmin{\Sigma_X(0)}}
,\,
\left(\frac{54 K}{\lmin{\Sigma_X(0)}}\right)^{\frac{2-\gamma}{1-\gamma}} \left(\frac{C_2}{C_1}\right)^{\frac{1}{1-\gamma}}
}
\end{align*}
we have with probability at least
\begin{align*}
 1-  2T \exp  \bcur{   -c' T^\gamma}, \text{ where } c' = \frac{\bpar{\lmin{\Sigma_X(0)}}^\gamma}{(54 K)^\gamma 2C_1}
\end{align*}
we have, for all $v \in \R^p$,
\begin{align*}
\hat{\Sigma}_T(v) &\ge  \alpha \vertii{v}^2_2   -\tau\vertii{v}^2_1 
\end{align*}
where
\begin{align*}
\alpha& = \half \lmin{\Sigma_X(0)}, & \tau &= \frac{\alpha}{2c'}  \cdot\bpar{ \frac{ \log(p)}{T^\gamma }} . \qedhere
\end{align*}
\end{proof}

%% file: VAR.tex
%!TEX root = aos-mixing.tex

Note that every VAR(d) process has an equivalent VAR(1) representation (see e.g. \cite[Ch 2.1]{lutkepohl2005new}) as 
 \begin{align}\label{eq:VAR(d)}
 \tilde{Z}_t = \tilde{\mt{A}}\tilde{Z}_{t-1} + \tilde{\mathcal{E}}_t
 \end{align}
  where 
 \begin{align}\label{eq:VAR1}
 \tilde{Z_t}:=
 \begin{bmatrix}
 Z_t \\ 
 Z_{t-1} \\
 \vdots \\
 Z_{t-d+1}
  \end{bmatrix}_{(pd \times 1)} 
  &
  \tilde{\mathcal{E}}_t:=
\begin{bmatrix}
\mathcal{E}_t \\ 
\vc{0} \\
 \vdots \\
\vc{0} 
\end{bmatrix}_{(pd \times 1)}  
\text{and }\;\;
\tilde{\mt{A}}:=
\begin{bmatrix}
\mt{A}_1 & \mt{A}_2 & \cdots & \mt{A}_{d-1} & \mt{A}_d \\ 
\mt{I}_p &  \vc{0}   &     \vc{0}    &      \vc{0}  & \vc{0}  \\ 
\vc{0}     & \mt{I}_p &        &   \vc{0}      & \vc{0}  \\ 
\vdots    &     & \ddots    &  	\vdots     & \vdots\\ 
\vc{0}     &  \vc{0}    &    \cdots    &     \mt{I}_p & \vc{0} 
\end{bmatrix}_{(dp \times dp)} 
 \end{align}
Because of this equivalence, justification of Assumptions \ref{as:alphaMix}(Gaussian case) and~\ref{assump:betamixing} (subweibull case) will operate through this corresponding augmented VAR$(1)  $ representation.

For both Gaussian and subweibull VARs, Assumption~\ref{as:0mean} is true since the sequences $ (Z_t) $ is centered. Second,  $ \bstar=(\mt{A}_1, \cdots, \mt{A}_d) $. So Assumption~\ref{as:spars}  follows from construction.

For the remaining assumptions, we will consider the Gaussian and subweibull cases separately.

\paragraph{Gaussian VAR}
$ (Z_t) $ satisfies Assumption~\ref{as:gauss} by model assumption.

To show that $({Z}_t)   $ is $ \alpha $-mixing with summable coefficients, we use the following facts together with the equivalence between $ (Z_t) $ and $ (\tilde{Z}_t) $ and Fact \ref{fact:mixingEquiv}.

Since $ (\tilde{Z}_t) $ is stable, the spectral radius of $ \tilde{\mt{A}} $, $r(\tilde{\mt{A}})<1  $, hence Assumption~\ref{as:stat} holds. Also the innovations $ \tilde{\mathcal{E}} $ has finite  first absolute moment and positive support everywhere. Then, according to   \citet[Theorem 4.4]{tjostheim1990non},  $ (\tilde{Z}_t) $ is \textit{geometrically ergodic}. Note here that Gaussianity is \emph{not} required here. Hence, it also applies to innovations from mixture of Gaussians.
 
Next, we present a standard result (see e.g. \cite[Proposition 2]{liebscher2005towards}). 
 
\begin{fact}\label{fact:ergodic}
A stationary Markov chain $\{\sv{Z}_t\}$ is geometrically ergodic implies 
 $\{\sv{Z}_t\}$ is \textit{absolutely regular} (or $ \beta $-mixing) with 
 $$
 \beta(n)=O(\gamma^n),\,\, \gamma \in(0,1)
 $$
\end{fact}
 
By the fact that $\beta$-mixing implies $\alpha$-mixing (see Section \ref{sec:mixingintro}) for a random process, we know that $ \alpha $-mixing coefficients decay geometrically and hence is summable. So, Assumption~\ref{as:alphaMix} holds. 

\paragraph{Subweibull VAR}
%When the innovations are random vectors from the uniform distribution, they are subweibull. 
%That $ (Z_t) $ are subweibull follows from arguments as in Appendix \ref{veri:ARCH} with $ \Sigma(\cdot) $ set to be the identity operator in this case. So, Assumption~\ref{assump:subW} holds. 

To show that $({Z}_t)   $ satisfies Assumptions~\ref{as:stat} and~\ref{assump:betamixing}, we establish that $({Z}_t)   $ is geometrically ergodic. To show the latter, we use Propositions 1 and 2 in~\cite{liebscher2005towards} together with the equivalence between $ (Z_t) $ and $ (\tilde{Z}_t) $ and Fact \ref{fact:mixingEquiv}. 
\edit{
It will be useful to note the fact that spectral radius $ r(\tilde{A}) <1 $ implies that $ \exists k \in \mathbf{Z} $ such that $ \vertiii{\tilde{A}^k } <1$.
}

To apply Proposition 1 in \cite{liebscher2005towards}, we check the three conditions one by one. 
Condition (i) is immediate with $ m=1,\, E=\R^p$, and $\mu$ is the Lebesgue measure. For condition (ii), we set $ E=\R^p$, $\mu$ to be the Lebesgue measure, and $\bar{m}= \ceil{\inf_{\vc{u} \in C, \vc{v} \in A}\vertii{u-v}_2} $ the minimum ``distance" between the sets $ C   $ and $ A $. Because $ C  $ is bounded and $ A $ Borel, $ \bar{m} $ is finite. Lastly, for condition (iii), we again let  $ E=\R^p$, $\mu$ to be the Lebesgue measure, and now the function $ Q(\cdot)= \vertii{\cdot} $ and then set 
\edit{
$ K_c=\{x \in \R^p: \vertii{x}\le \frac{4C_{A\epsilon}}{c} \} $ where $ c=1-\vertiii{\tilde{A}^k} $ and $ C_{A\epsilon} := \sum_{i=1}^{k+1} \vertiii{\tilde{A}^{k-i}} \E \vertii{\epsilon_{t-k+1}} $. Then, since spectral radius $ r(\tilde{A}) <1 $,}

\begin{itemize}
\item %Recall from model assumption that $ \vertiii{\tilde{\mt{A}}}<1 $;  hence, 
\edit{
For all $ z \in E \backslash K_c  $; i.e. $ z  $ such that $ \vertii{z} > \frac{4C_{A\epsilon}}{c} $,
\begin{align*}
\E \bbra{\vertii{\,\tilde{Z}_{t+1} }\,\middle| \tilde{Z}_t=z} &\le \vertiii{\tilde{A}^k} \vertii{z} + \sum_{i=1}^{k+1} \vertiii{\tilde{A}^{k-i}} \E \vertii{\epsilon_{t-k+1}} \\
&\equiv \bpar{1-c}\vertii{z} + C_{A\epsilon} \\
&< \bpar{1-\frac{c}{2}}\vertii{z} -C_{A\epsilon} .\; 
%&\text{for all } z \in E\backslash K_c
\end{align*} 
%where $ \epsilon := C_{A\epsilon}   $.\\
}
\item
\edit{
For all $ z \in K_c $,
\begin{align*}
\E \bbra{\vertii{\,\tilde{Z}_{t+1} }\,\middle| \tilde{Z}_t=z} <  \vertiii{\tilde{{A}}} \vertii{z} + C_{A\epsilon}  
\le \frac{4C_{A\epsilon} (1-c)}{c} + C_{A\epsilon}
\end{align*}
}
\item
\edit{
For all $ z \in K_c $,
\begin{align*}
0 \le \vertii{z } \le \frac{4C_{A\epsilon}}{c}.
\end{align*}
}
\end{itemize}

Now, by Proposition 1 in~\cite{liebscher2005towards}, $ (\tilde{Z}_t) $ is geometrically ergodic; hence $ (\tilde{Z}_t) $ will be stationary. Once it reaches stationarity, by Proposition 2 in the same paper, the sequence will be $ \beta $-mixing with geometrically decaying mixing coefficients. 
%Since $ \beta $-mixing condition is stronger than $ \alpha $-mixing, we have geometrically decaying and hence summable $ \alpha $-mixing coefficients as well. 
Therefore, Assumptions~\ref{as:stat} and~\ref{assump:betamixing} hold.

We are left with checking Assumption~\ref{assump:subW}. Let $ \gamma $ be the subweibull parameter associated with $ (\mathcal{E}_t)$. 

%By recursive nature of the time series,
%\begin{align*}
%\norm[Z_t]{\psi_\gamma} \le \vertiii{A} \norm[Z_{t-1}]{\psi_\gamma} +  \norm[\mathcal{E}_{t}]{\psi_\gamma}
%\end{align*}
%which yields
%\begin{align*}
%\norm[Z_t]{\psi_\gamma} \le \frac{\norm[\mathcal{E}_{t}]{\psi_\gamma}}{1- \vertiii{A}} <\infty
%\end{align*}

Assume that the spectral radius of $ A $ is smaller than $ 1 $; i.e. $ r(A) < 1 $. This is an equivalent notion of stability of VAR process. 
%Note that if we instead require $ \vertiii{A}<1 $, it will be a much more stringent condition, and is in fact vacuous for VAR(d), where $ d>1 $ (see for example~\cite{basu2015regularized}).
By the definition of the spectral radius,
$$ \lim\limits_{m \to \infty}  \vertiii{A^m}^{1/m} = r(A)<1$$ In other words, there exists a positive integer $ k <\infty $ such that $  \vertiii{\tilde{A}^k}<1 $.
By the recursive nature of the time series,
\begin{align} \label{subweibull-norm-app}
\norm[Z_t]{\psi_\gamma} \le \vertiii{\tilde{A}^k} \norm[Z_{t-1}]{\psi_\gamma} +  \sum_{i=1}^{k+1}\vertiii{\tilde{A}^{k-i}}\norm[\mathcal{E}_{t-k+i}]{\psi_\gamma}
\end{align}
To simplify notation, let $ C_i := \vertiii{\tilde{A}^{k-i}} $. Using stationarity, we have the following
\begin{align*}
\norm[Z_t]{\psi_\gamma} \le \frac{\swnorm{\epsilon_t}}{1-\vertiii{\tilde{A}^k}} \bpar{\sum_{i=1}^{k}c_i} < \infty
\end{align*}
The last inequality follows because $ C_i<\infty, \, \forall i=1, \cdots, k$. 
Thus, the sequence $ (Z_t) $ satisfies Assumption~\ref{assump:subW}.

%% file: misspecifiedVAR.tex
%!TEX root = aos-mixing.tex

%\textbf{Goal}: We will study OLS estimator of a VAR(1) process when one endogenous variable is left out. This arises naturally when the underlying DGM is high-dimensional but not all variables are available/observable/measurable to the researcher to do estimation/prediction. This also happens when the researcher mis-specifies the scope of the model. Notice that the system of the restricted set of variables is no longer a finite order VAR. There is model mis-specification and this example serves to illustrate that our theory is applicable to models beyond VAR setting.  \\
%
%
%\paragraph{Setup:}
%$ $\\
%Consider a VAR(1) process:
%$$
%(Z_t, \Xi_t)' = A (Z_{t-1}, \Xi_{t-1})' + (\mathcal{E}_{Z, t-1}, \mathcal{E}_{\Xi, t-1})'
%$$
%Where $Z_t \in \R^{p} $, $\Xi_t \in \R^{1} $ are partition of the random vector $ (Z_t, \Xi_t) $ in the VAR(1) process, and
%$$
%A:= 
%\bbra{
%\begin{array}{cc}
%A_{ZZ} & A_{Z\Xi} \\ 
%A_{\Xi Z} & A_{\Xi \Xi}
%\end{array}  
%}
%$$
%is the transition/coefficient matrix of the VAR(1) process. \\

%\textbf{Assumptions}:
Assumption~\ref{as:0mean} is immediate from model definitions. By the same arguments as in Appendix \ref{veri:VAR}, $ (Z_t,\Xi_t) $ are stationary and so is the sub-process $ (Z_t) $; Assumption~\ref{as:stat} holds. Again,   $ (Z_t,\Xi_t) $ 
%in both Examples \ref{ex:misVAR} and \ref{ex:sGmisVAR} 
satisfy  Assumption \ref{as:alphaMix} (for Example \ref{ex:misVAR}) and Assumption~\ref{assump:betamixing} (for Example \ref{ex:sGmisVAR}) according to Appendix \ref{veri:VAR}. By Fact \ref{fact:mixingEquiv}, we have the same Assumptions hold for the respective sub-processes $ (Z_t) $ in the respective cases. 
%\begin{enumerate}
%\item  \label{indep}
%Independent errors; i.e. $\mathcal{E}_{X, t} \indep \mathcal{E}_{Z, t'},\; \forall t,t'$, $\mathcal{E}_{Z, t'} \indep
%(Z_t,\Xi_t),\; \forall t'\ge t$ and $\mathcal{E}_{X, t'} \indep
%(Z_t,\Xi_t),\; \forall t'\ge t$.
%\item  \label{stat}
%Stationarity (2nd order); i.e. $\Sigma(h):= \E(Z_t, \Xi_t)(Z_{t+h}, Z_{t+h})' $ are independent of time for all $h\in \mathbb{Z}$ and $ \E(Z_t, \Xi_t) =\vc{0}\; \forall t$. 
%%\item \label{exc}
%%$ A_{12} $ is $ 1 $-sparse. I.e.  $\vertii{(\vect{A_{12})}}_0=1$.
%\end{enumerate}

To show that $ (\bstar)'=\mt{A}_{Z Z }+\mt{A}_{Z \Xi } \Sigma_{\Xi Z }(0)(\Sigma_Z (0))^{-1} $, consider the following arguments. 
By Assumption~\ref{as:stat}, we have
%we can define $ \forall h \in \mathbb{Z} $, $\Sigma_{XZ}(h) = \E(Z_t)(Z_{t+h}) $, $\Sigma_{ZX}(h) = \E(\Xi_t)(Z_{t+h})' $,  $\Sigma_X(h):= \E(Z_t)(Z_{t+h})' $ and $\Sigma_Z(h):= \E(\Xi_t)(Z_{t+h}) $. Hence 
the auto-covariance matrix of the whole system $ (Z_t,\Xi_t) $ as

\[
\Sigma_{(Z,\,\Xi)}
=
\begin{bmatrix}
\Sigma_X(0) &  \Sigma_{X\Xi}(0)\\ 
\Sigma_{\Xi X}(0) & \Sigma_\Xi (0)
\end{bmatrix} 
\]
Recall our $ \bstar $ definition from Eq.~\eqref{eqn:bstar}
$$
\bstar := \argmin_{\mt{B} \in \R ^{p\times p}} \E \bpar{  
\vertii{
Z _t - \mt{B}' Z _{t-1}
}^2_2
}\\
$$
Taking derivatives and setting to zero, we obtain
\begin{equation} \label{eq:tilA}
(\bstar)' = \Sigma_Z (-1) (\Sigma_Z )^{-1}
\end{equation}
Note that 
\begin{align*}
\Sigma_Z (-1)	&= \Sigma_{(Z ,\,\Xi )}(-1)[1:p_1, 1:p_1] \\
				&= \E \bpar{\mt{A}_{Z Z }Z _{t-1} + \mt{A}_{Z \Xi }\Xi _{t-1} + \mathcal{E}_{Z ,t-1}} Z_{t-1}' \\
				&=\E\bpar{\mt{A}_{Z Z }Z_{t-1} Z_{t-1}' +  \mt{A}_{Z \Xi }\Xi _{t-1}  Z _{t-1}' +\mathcal{E}_{Z ,t-1}  Z_{t-1}' }\\
				&= \mt{A}_{Z Z }\Sigma_Z (0) +\mt{A}_{Z \Xi } \Sigma_{\Xi Z }(0)
\end{align*}
by Assumption \ref{as:stat} and the fact that the innovations are iid.

Naturally, 
$$
(\bstar)'
= \mt{A}_{Z Z }\Sigma_Z (0)(\Sigma_Z (0))^{-1} +\mt{A}_{Z \Xi } \Sigma_{\Xi Z }(0)(\Sigma_Z (0))^{-1}
=\mt{A}_{Z Z }+\mt{A}_{Z \Xi } \Sigma_{\Xi Z }(0)(\Sigma_Z (0))^{-1}
$$

\begin{rem}
Notice that $\mt{A}_{Z \Xi }   $ is a column vector and suppose it is $ 1 $-sparse, and $ \mt{A}_{Z Z } $ is $ p $-sparse, then $ \bstar$ is at most $  2p$-sparse. So Assumption~\ref{as:spars} can be built in by model construction. 
\end{rem}

\begin{rem}
We gave an explicit model here where the left out variable $ \Xi  $ was univariate. That was only for convenience. In fact, whenever the set of left-out variables $ \Xi  $ affect only a small  set of variables $ \Xi  $ in the retained system $ Z  $, the matrix $\bstar$ is guaranteed to be sparse. To see that, suppose $ \Xi \in \R^q $ and $ \mt{A}_{Z \Xi } $ has at most $s_0$ non-zero rows (and let $ \mt{A}_{Z Z } $ to be $ s $-sparse as always), then $\bstar$ is at most $ (s_0 p +s )$-sparse.
\end{rem}
Lastly, for Example~\ref{ex:misVAR}, the sub-process $ (Z_t) $ is Gaussian because is obtained from a linear transformation of $ (Z_t,\Xi_t) $ which is Gaussian; we have Assumption~\ref{as:gauss}. For Example~\ref{ex:sGmisVAR}, note that $ Z_t = \mt{M} (Z_t, {\Xi}_t )$ where $ \mt{M}= [\mt{I}_p,0;\vc{0}',0] $ is a sub-setting matrix that selects the first $ p $ entries of a $ (p+1) $-dimensional vector. Hence, the fact that $Z_t $ is subweibull follows from the same arguments in Appendix~\ref{veri:VAR} pertaining to establishing the subweibull property in conjunction with applying Lemma~\ref{result:subweibulllinear} on $ Z_t = \mt{M} (Z_t, {\Xi}_t )$; so, Assumption~\ref{assump:subW} holds.  
%\begin{enumerate}
%\item
%Gaussian VAR: This is the vanilla gaussian VAR process where $(\mathcal{E}_{\Xi , t-1}, \mathcal{E}_{\Xi, t-1}) \sim MVN(\vc{0}, \Sigma_{\mathcal{E}})$
%\item
%Sub-gaussian VAR: If $(\mathcal{E}_{\Xi , t-1}, \mathcal{E}_{\Xi, t-1})[i] \overset{i.i.d.}{\sim} \U{[-K,K]}, \forall i$ for some $K>0$. When the innovations are random vectors from mixture of Gaussian, they are sub-Gaussian. That $ (Z_t) $ are sub-Gaussian follows from arguments as in Appendix \ref{thm:hanson} with $ \Sigma(\cdot) := I_{p \times p} $ in this case. So, Assumption~\ref{assum:subgauss-} holds. 
%\end{enumerate}
\begin{rem}
Any VAR($d$) process has an equivalent VAR(1) representation~\cite{lutkepohl2005new}. Our results extend to any VAR($d$) processes. 
\end{rem}

%\paragraph{References}
%\begin{enumerate}
%\item
%Lütkepohl, Helmut. New introduction to multiple time series analysis. Springer Science $\&$ Business Media, 2005.
%\end{enumerate}

%% file: MutltivariateARARCH.tex
\paragraph{Verifying the Assumptions}

To show that Assumption~\ref{assump:betamixing} hold for a process defined by Eq. \eqref{eq:ARCH} we leverage on Theorem 2 from~\cite{liebscher2005towards}. 
Note that the original ARCH model in~\cite{liebscher2005towards} assumes the innovations to have positive support everywhere. However, this is just a convenient assumption to establish the first two conditions in Proposition 1 (on which proof of Theorem 2 relies) from the same paper. ARCH model with innovations from  more general distributions (e.g. uniform) also satisfies the first two conditions of Proposition 1 by the same arguments  in the \emph{Subweibull} paragraph of Appendix~\ref{veri:VAR}. 

Theorem 2 tells us that for our ARCH model, if it satisfies the following conditions, it is guaranteed to be 
%geometrically ergodic ($Q$-geometric ergodic implies geometric ergodic when the $Q$ function exists, which is guaranteed by theory here) and 
absolutely regular with geometrically decaying $ \beta $-coefficients.
\begin{itemize}
\item $\mathcal{E}_t$ has positive density everywhere on $\R^p$ and has identity covariance by construction.
\item $\Sigma(\vc{z}) = o(\vertii{\vc{z}})$ because $m \in (0,1)$.
\item $\vertiii{ \Sigma(\vc{z})^{-1} } \le 1/(ac)$, $|\mathrm{det} \left( \Sigma(\vc{z}) \right) | \le bc$
\item \edit{ $r(\mt{A}) < 1$}
\end{itemize}
So, Assumption~\ref{assump:betamixing} is valid here. We check other assumptions next.

Mean $ 0 $ is immediate, so we have Assumption~\ref{as:0mean}.
When the Markov chain did not start from a stationary distribution, geometric ergodicity implies that the sequence is approaching the stationary distribution exponentially fast. So, after a burning period, we will have Assumption \ref{as:stat} approximately valid here. 
%In addition, absolute regularity implies exponentially decaying $\beta$-coefficients, so condition \ref{assum:absReg} is fulfilled. 

%\red{TODO: deal with the inverse issue belo w}
The subweibull constant of $\Sigma(Z_{t-1})\mathcal{E}_t$ given $ Z_{t-1}=\vc{z} $ is bounded as follows:
for every $ \vc{z} $,
%\begin{align*}
%\swnorm{ \Sigma(\vc{z})\mathcal{E}_t } &\le \vertiii{ \Sigma(\vc{z})^{-\half} } \swnorm{\mathcal{E}_t} \\
%&\le C\vertiii{ \Sigma(\vc{z})^{-\half} }  \cdot  \swnorm{e_1' \mathcal{E}_t} \\
%&\le C \vertiii{ \Sigma(\vc{z})^{-\half} }  \cdot  \swnorm{\U{-\sqrt{3},\sqrt{3}}}  \\
%&\le C'/(\sqrt{ac})=: K_E
%\end{align*}

\begin{align*}
\swnorm{ \Sigma(\vc{z})\mathcal{E}_t } &\le \vertiii{ \Sigma(\vc{z}) } \swnorm{\mathcal{E}_t} &&\text{by Lemma~\ref{result:subweibulllinear}}\\
%&\le C \vertiii{ \Sigma(\vc{z}) }  \cdot  \swnorm{e_1' \mathcal{E}_t} \\
%&\le C  \vertiii{ \Sigma(\vc{z}) }  \cdot  \swnorm{\U{-\sqrt{3},\sqrt{3}}}  \\
&\le K_e cb=: K_E
\end{align*}
where $ K_e := \sup\limits_{t}  \swnorm{\mathcal{E}_t}  $. 
%The second inequality follows since $ \mathcal{E}_t \overset{iid}{\sim} \U{\bbra{-\sqrt{3},\sqrt{3}}^p} $ and a standard result that
%\begin{fact}
%Let $ X=(X_1,\cdots,X_p)\in \R^p $ be a random vector with independent, mean zero, subweibull coordinates $ X_i $. Then $ X $ is a subweibull random vector, and there exists a positive constant $ C $ for which
%\begin{align*}
%\swnorm{X}\le C \cdot \max_{i\le p}\swnorm{X_i}
%\end{align*}
%\end{fact}
%The forth inequality follows since the subweibull norm of a bounded random variable is also bounded. 
\edit{
By the same arugments as in Equation~\ref{subweibull-norm-app}, we have that
%By the recursive relationship of $(Z_t)_t$, we have
%\[
%\swnorm{Z_t} \le \vertiii{ \mt{A} } \swnorm{Z_{t-1}} + K_E .
%\]
%which yields the bound $\swnorm{Z_t} \le  K_E / (1-\vertiii{ \mt{A} } ) < \infty$. Hence 
Assumption~\ref{assump:subW} holds.
}

We will show below that $\bstar=\mt{A}'$. Hence, sparsity (Assumption~\ref{as:spars}) can be built in when we construct our model \ref{eq:ARCH}. 

Recall Eq. \ref{eq:tilA} from Appendix~\ref{veri:misVAR} that 
$$
\bstar= \Sigma_Z(-1) (\Sigma_Z)^{-1}
$$
Now,
\begin{align*}
\Sigma_Z(-1)&=\E Z_t Z_{t-1}'&&\text{by stationarity} \\
			&= \E \bpar{  \mt{A} Z_{t-1} +\Sigma(Z_{t-1}) \mathcal{E}_t }Z_{t-1}' &&\text{Eq.~\eqref{eq:ARCH}}\\
			&=\mt{A} \E Z_{t-1}Z_{t-1}' + \E \Sigma(Z_{t-1}) \mathcal{E}_t Z_{t-1 }' \\
			&= \mt{A}\Sigma_Z  + \E[ c\,\clip{\vertii{Z_{t-1}}^{m}}{a}{b}  \mathcal{E}_t Z_{t-1 }'] \\
			&= \mt{A}\Sigma_Z  +  \E[ c\mathcal{E}_t Z_{t-1 }' \clip{\vertii{Z_{t-1}}^{m}}{a}{b}]\\
			&= \mt{A}\Sigma_Z  +   c\E\bbra{ \mathcal{E}_t} \E\bbra{ Z_{t-1 }'  \clip{\vertii{Z_{t-1}}^{m}}{a}{b} } &&\text{i.i.d. innovations}\\
			&=\mt{A}\Sigma_Z  &&\text{$ \mathcal{E}_t $ mean $ 0 $} ,
\end{align*}
where $\clip{x}{a}{b} := \min\{\max\{x,a\},b\}$ for $b > a$.

Since $ \Sigma_Z $ is invertible, we have $ (\bstar)'= \Sigma_Z(-1) (\Sigma_Z)^{-1} =\mt{A} $.